\documentclass{psp}
\usepackage{amsmath,amssymb,bm}
\usepackage{graphicx}
\usepackage[usenames,dvipsnames]{color}
\usepackage{float}
\usepackage{graphicx}
\usepackage{subfigure}
\usepackage{comment}
\usepackage{cite}
\usepackage{url}
\usepackage{mathrsfs} 
\usepackage{epstopdf}
\usepackage{color}

\ifprodtf \else
\checkfont{msam10}
\iffontfound
\IfFileExists{amsfonts.sty}
{\typeout{^^JFound AMS Symbol fonts on the system, using the
		'amsfonts' package.^^J}%
	\usepackage{amsfonts}
}
{\providecommand\mathbb[1]{\mathsf{##1}}
	\providecommand\mathfrak[1]{\mathcal{##1}}}
\else
\providecommand\mathbb[1]{\mathsf{#1}}
\providecommand\mathfrak[1]{\mathcal{#1}}
\fi

\fi

\ifprodtf \else
\IfFileExists{amsbsy.sty}
{\typeout{^^JFound the 'amsbsy' package on the system, using it.^^J}%
	\usepackage{amsbsy}}
{}

\fi

\newcommand{\R}{\mathbb{R}}

\newcommand{\dint}{\displaystyle\int}

\newtheorem{theorem}{Theorem}[section]
\newtheorem{corollary}[theorem]{Corollary}
\newtheorem{proposition}[theorem]{Proposition}
\newtheorem{definition}[theorem]{Definition}
\newtheorem{example}{Examples}[section]
\newtheorem{remark}[theorem]{Remark}
\newtheorem{lemma}[theorem]{Lemma}

\newromanexpr\Hess{Hess}

\begin{document}
	
	\title[]{Fractional Brownian motion with deterministic drift: How critical is drift regularity in hitting probabilities}
		\author[Mohamed Erraoui and  Youssef Hakiki]{MOHAMED ERRAOUI\\
		 Department of mathematics, Faculty of science El jadida,\\ Chouaïb Doukkali University,
		Morocco \\
	e-mail\textup{: \texttt{erraoui@uca.ac.ma}}
       \and\space  YOUSSEF HAKIKI\\
           Department of mathematics, Faculty of science Semlalia,\\ Cadi Ayyad University, 2390 Marrakesh, Morocco \\
           e-mail\textup{: \texttt{youssef.hakiki@ced.uca.ma}}}

 \maketitle

	\begin{abstract}
Let $B^{H}$ be a $d$-dimensional fractional Brownian motion
with Hurst index $H\in(0,1)$, $f:[0,1]\longrightarrow\mathbb{R}^{d}$
a Borel function, and $E\subset[0,1]$, $F\subset\mathbb{R}^{d}$
are given Borel sets. The focus of this paper is on hitting probabilities of the non-centered Gaussian process $B^{H}+f$. It aims to highlight how each component $f$, $E$ and $F$ is involved in determining the upper and lower bounds of $\mathbb{P}\{(B^H+f)(E)\cap F\neq \emptyset \}$. When $F$ is a singleton and $f$ is a general measurable drift, some new estimates are obtained for the last probability by means of suitables Hausdorff measure and capacity of the graph $Gr_E(f)$. As application we deal with the issue of polarity of points for $(B^H+f)\vert_E$ (the restriction of $B^H+f$ to the subset $E\subset (0,\infty)$).
\end{abstract}


\vspace{0,2cm}



	\section{Introduction}

Hitting probabilities describes the probability that a given stochastic process
will ever reach some state or set of states $F$. To find upper and lower bounds for the hitting probabilities in terms of the Hausdorff measure and the capacity of the set $F$, is a fundamental problem in probabilistic potential theory. For $d$-dimensional Brownian motion $B$, the
probability that a path will ever visit a given set $F\subseteq\mathbb{R}^{d}$,
is classically estimated using the Newtonian capacity of $F$. Kakutani \cite{Kak44}
was the first to establish this result linking capacities and hitting
probabilities for Brownian motion. He showed that, for all $x\in\mathbb{R}^{d}$  and compact
set $F$ 
\[
\mathbb{P}_{x}\{B(0,\infty)\cap F\neq\emptyset\}>0 \quad  \Longleftrightarrow \quad \text{Cap}\left(F\right)>0,
\]
where $\text{Cap}$ denotes the capacity associated to the Newtonian and logarithmic kernels ($\left|\cdot\right|^{2-d}$, if $d\geq 3$  and $\left|\log\left(1/\cdot\right)\right|$, if $d=2$ ) respectively. The similar problem has been considered by Peres and Sousi \cite{Peres Sousi} for $d$-dimensional Brownian motion $B$ with $d\geq2$ with a drift function $f$.
They showed that for a $(1/2)$-Hölder continuous function $f:\mathbb{R}^{+}\longrightarrow\mathbb{R}^{d}$ there is a positive constant $c_2$ such that for all $x\in\mathbb{R}^{d}$ and all closed set $F\subseteq\mathbb{R}^{d}$ 
\begin{equation}\label{Per-Sou}
\mathsf{c}^{-1}_{1}\text{Cap}_{M}\left(F\right)\leq\mathbb{P}_{x}\{(B+f)(0,\infty)\cap F\neq\emptyset\}\leq \mathsf{c}_{1}\text{Cap}_{M}\left(F\right),
\end{equation}
where $\text{Cap}_{M}\left(\cdot\right)$ denotes the Martin capacity, see for example \cite{MP}. At the heart of their method is the strong Markov property.
It is noteworthy that since Kakutani's characterization, considerable efforts have been carried out to establish
a series of extensions to other processes, on the one hand, and, secondly, for a restricted subset $E\subset (0,\infty)$. This has given rise to
a large and rapidly growing body of scientific literature on the subject.
To cite a few examples, we refer to Xiao \cite{Xiao 1996} for developments on hitting probabilities of stationary
Gaussian random fields and fractional Brownian motion; to Pruitt and Taylor \cite{Pruitt Taylor} and Khoshnevisan \cite{Khoshnevisan 1997} for hitting probabilities
results for general stable processes and Lévy processes; to Khoshnevisan and Shi \cite{Khoshnevisan shi 1999} for hitting probabilities of the Brownian sheet; to Dalang and Nualart
\cite{Dalang Nualart 2004} for hitting probabilities for the solution of a system of nonlinear
hyperbolic stochastic partial differential equations;
to Dalang, Khoshnevisan and Nualart \cite{Dalang Khoshnevisan Nualart 2007} and \cite{Dalang Khoshnevisan Nualart 2009},  for hitting probabilities for the solution of a 
non-linear stochastic heat equation with additive and multiplicative noise respectively; to Xiao \cite{Xiao 2009}  Biermé,
Lacaux and Xiao \cite{Bierme Lacaux Xiao 2009} for hitting probabilities of anisotropic Gaussian random fields. 

An important remark should be made, Kakutani's characterization is not common to all the processes and this is generally due to process dependency structures. Thus the lower and upper bounds on hitting probabilities are obtained respectively in terms of capacity and Hausdorff measure of the product set $E\times F$. In this light, we cite Chen and Xiao result \cite{Chen Xiao 2012} which is actually an improvement of results established by Xiao 
\cite[Theorem 7.6]{Xiao 2009} and by Biermé, Lacaux and Xiao {\cite[Theorem 2.1]{Bierme Lacaux Xiao 2009}} on hitting probabilities of the $\mathbb{R}^{d}$-valued Gaussian
random field $X$ satisfying conditions $(C_{1})$ and $(C_{2})$  (see Xiao \cite{Xiao 2009} for precise definition) through the following estimates 
\begin{equation*}
\mathsf{c}_2^{-1}\mathcal{C}^d_{\rho_{\mathbf{d}_X}}(E\times F)\leq\mathbb{P}\{X(E)\cap F\neq\emptyset\}\leq \mathsf{c}_2\mathcal{H}_{\rho_{\mathbf{d}_{X}}}^{d}(E\times F),\label{Chen Xiao Estimate}
\end{equation*}
where $E\subseteq[a,1]^{N}$, $a\in(0,1)$
and $F\subseteq\mathbb{R}^{d}$ are Borel sets and $\mathsf{c}_2$ is a
constant which depends on $[a,1]^{N}$, $F$ and $H$ only. $\mathcal{C}^d_{\rho_{\mathbf{d}_{X}}}(\cdot)$ and $\mathcal{H}_{\rho_{\mathbf{d}_{X}}}^{d}(\cdot)$ denotes the Bessel-Riesz type capacity and the Hausdorff measure with respect to the metric $\rho_{{\mathbf{d}_{X}}}$, defined on $\mathbb{R}_+\times\mathbb{R}^d$ by $\rho_{{\mathbf{d}_{X}}}\left((s,x),(t,y)\right):=\max \left\{\mathrm{d}_X(s,t),\lVert x-y\rVert \right\}$ where $\mathrm{d}_X$ is the canonical metric of the Gaussian process $X$. Both of these terms are defined in the sequel.
 We emphasize that fractional Brownian motion belongs to the class of processes that satisfy the conditions $(C_{1})$ and $(C_{2})$. See Xiao \cite{Xiao 2009} for information on the conditions $(C_{1})$ and $(C_{2})$.

The similar issue for a multifractional Brownian motion $B^{H(\cdot)}$ governed by a Hurst function
$H(\cdot)$ and drifted by a function $f$ with the same Hölder exponent function $H(\cdot)$ was investigated by Chen \cite[Theorem 2.6]{Chen2013}. The hitting probabilities estimates upper and lower 
were given in terms of $\mathcal{H}^d_{\rho_{\mathrm{d}_{\underline{H}}}}(E\times F)$ and $\mathcal{C}_{\rho_{\mathrm{d}_{\overline{H}}}}^d(E\times F)$ respectively. The metrics $\mathrm{d}_{\underline{H}}$ and $\mathrm{d}_{\overline{H}}$ are defined by $|t-s|^{\underline{H}}$ and $|t-s|^{\overline{H}}$ respectively, where  $\underline{H}$ and $\overline{H}$ are the extreme values of the Hurst function $H(\cdot)$. Consequently, for the fractional Brownian $B^H$ of Hurst parameter $H$ (H is now time-independent), Chen's result can be formulated as follows
\begin{equation}\label{Chen estimate}
\mathsf{c}_3^{-1}\,\mathcal{C}^d_{\rho_{\mathbf{d}_{H}}}(E\times F)\leq\mathbb{P}\{(B^H+f)(E)\cap F\neq\emptyset\}\leq \mathsf{c}_3
\mathcal{H}_{\rho_{\mathbf{d}_{H}}}^{d}(E\times F),
\end{equation}
for any $H$-Hölder continuous drift $f$. {What we notice in both above-mentioned results \eqref{Per-Sou} and \eqref{Chen estimate} is that the drift $f$ has $H$ as the Hölder exponent which is out of reach for fractional Brownian paths. This brings us to consider the sensitivity problem of the above estimates with respect to the Hölder exponent of the drift.
More precisely, could we obtain the same results when the drift $f$ is $(H-\varepsilon)$-Hölder continuous for any $\varepsilon>0$ small enough?}

Our first objective in this work is to give an answer to this issue. First of all, we seek to provide some general results on Hausdorff measures and capacities that will be relevant to reaching our main goal. Precisely, we established some upper bounds for the Hausdorff measure (resp. lower bounds for the Bessel-Riesz capacity) of the product set $E\times F$ by means of a suitable Hausdorff measure (resp. Bessel-Riesz capacity) of one of the component sets $E$ and $F$ in a general metric space. If either $E$ or $F$ is "reasonably regular" in the sense of having equal Hausdorff and Minkowski dimensions, the bounds become as sharp as possible. This constitutes the content of Section 2.   

In consequence with the foregoing, we obtain practical and appropriate bound estimates of hitting probabilities for fractional Brownian motion with $H$-Hölder continuous drift $B^H+f$ involving only Hausdorff measure (for upper bound) and Bessel-Riesz capacity (for lower bound) of $E$ or $F$, which we believe are of independent interest. All of this and related details make up the content of Section 3.

With the contents of the aforementioned sections, in Section 4 we address the sensitivity question posed above and the answer is negative. The idea is to look for drifts, linked with potential theory, that are Hölder continuous without reaching the Hölder regularity of order $H$. To be more specific, we seek drifts with a modulus of continuity of the form $\mathsf{w}(x)= x^H\,\ell(x)$, where $\ell$ is a slowly varying function at zero in the sens of Karamata satisfying $\underset{x\rightarrow 0}{\limsup}\,\ell(x)=+\infty$, and having close ties to hitting probabilities.
%
%
The potential candidates that best match the requested features are the paths of Gaussian processes with covariance function satisfying
\begin{align*}\label{delta compar}
     \mathbb{E}\left(B^{\delta_{\theta_H}}_0(t)-B^{\delta_{\theta_H}}_0(s)\right)^2=\delta_{\theta_H}^2(|t-s|) \quad\text{ for all $t,s\in [0,1]$},
\end{align*}
for some regularly varying function function $\delta_{\theta_H}(r):= r^H\, \ell_{\theta_H}(r)$ such that its slowly varying part $\ell_{\delta_{\theta_H}}$ satisfies $\underset{r\rightarrow 0}{\limsup}\,\ell(r)=+\infty$. Indeed, on the one hand the paths of such Gaussian processes are continuous with an almost sure uniform modulus of continuity given by $\delta_{\theta_H}(r)\log^{1/2}(1/r)$ up to a deterministic constant and, most of all, we are going to take advantage of the hitting estimates already established for such processes on the other. Based on a result due to Taylor \cite{Taylor61} on the relationship between Hausdorff measure and capacity and using the estimates established in Section 2, we construct two sets $E$ and $F$ and a drift $f$, chosen amongst the candidates listed above, for which the inequalities in \eqref{Chen estimate} are not checked.

 \noindent From this follows another question which will be dealt with in Section 5: What about the hitting probabilities for general measurable drift?
 
 \noindent In view of the lack of information about the roughness of the drift $f$ it is difficult to carry out upper and lower bounds on hitting probabilities even for regular sets $E$ and $F$. We were only able to tackle the issue of hitting probabilities when F is a single point, for a thorough explanation see Remark \ref{hard extension}. The idea is to use information provided by the graph $Gr_E(f)=\{(t,f(t)) : t\in E\}$ of $f$ over the Borel set $E$ in order to investigate the hitting probabilities. We establish, for a general measurable drift, upper and lower bounds on hitting probabilities of the type
 \begin{equation}\label{Hak-Err estimates}
\mathsf{c}_4^{-1}\mathcal{C}^{d}_{\rho_{\mathbf{d}_{H}}}(Gr_E(f))\leq \mathbb{P}\left\lbrace \exists t\in E : (B^H+f)(t)=x\right\rbrace\leq \mathsf{c}_4\, \mathcal{H}_{\rho_{\mathbf{d}_{H}}}^{d}(Gr_E(f)).
\end{equation} 
In connection with our previous work \cite{Erraoui Hakiki} on whether the image $(B^H + f) (E)$ has a positive Lebesgue measure $\lambda_d(B^H + f)(E)$ we have relaxed, thanks to the estimation above, the main assumption $\dim_{\rho_{\mathbf{d}_{H}}}(Gr_E(f))>d$ in \cite[Theorem 3.2]{Erraoui Hakiki}. Indeed, we show that the weaker condition $\mathcal{C}_{\rho_{\mathbf{d}_{H}}}^{d}(Gr_E(f))>0$ is sufficient to provide $\lambda_d\left(B^H+f\right)(E)>0$ with positive probability.
 
 Section 6 is devoted to stress the link between the Hausdorff dimension $\dim(E)$ of $E$ and the polarity of points for $(B^H+f)\vert_E$ (the restriction of $B^H+f$ to the subset $E\subset (0,\infty)$). This will be done through the above estimates \eqref{Chen estimate} (with $F$ a single point) and \eqref{Hak-Err estimates} as we explain in more detail below. First when $f$ is $H$-Hölder continuous the above estimates \eqref{Hak-Err estimates} take the following form
\begin{equation}\label{hitt pts particular case}
\mathsf{c}_5^{-1}\, \mathcal{C}^{Hd}(E)\leq \mathbb{P}\left\lbrace \exists t\in E : (B^H+f)(t)=x\right\rbrace\leq \mathsf{c}_5\, \mathcal{H}^{Hd}(E).
\end{equation}
Hence the first conclusion that we can draw is: if $\dim(E)>Hd$ then $(B^H+f)\vert_E$ hits points, and if $\dim(E)<Hd$ then $(B^H+f)\vert_E$ doesn't hit points, i.e. points are polar for $(B^H+f)\vert_E$. It is thus obvious that  $\dim(E)=Hd$ is the critical dimension case. We show that this case is undecidable, in the sense that we construct two Borel sets $E_1,E_2\subset [0,1]$ such that $\dim(E_1)=\dim(E_2)=Hd$ for which $\left(B^H+f\right)\lvert_{E_1}$ hits points but $\left(B^H+f\right)\lvert_{E_2}$ doesn't. This suggest that the roughness provided by the $H$-Hölder continuity of the drift $f$ is insufficient to allow $(B^{H}+f)\lvert_E$ to hits points for general set $E$ in the critical dimension case and as foreseeable question: could this be done by adding little more roughness? The idea is to take advantage of the lower bound of the hitting probabilities in \eqref{Hak-Err estimates} by looking for drifts for which $\mathcal{C}^{d}_{\rho_{\mathbf{d}_{H}}}(Gr_E(f))>0$ even for the critical case  $\dim(E)=Hd$. Again, as before, we use independent Gaussian process of type $B^{\delta_{\theta_H}}$ to construct drifts with slightly more roughness than that provided by $H$-Hölder continuity and for which the last condition must be satisfied.
With an appropriate choice of ingredients, especially the kernel defining the capacity with respect to the potential theory associated to $B^{\delta_{\theta_H}}$, to satisfy the assumptions of Taylor's result \cite{Taylor61} we construct drifts $f$, that are $(H-\varepsilon)$-Hölder continuous for all small $\varepsilon>0$ without reaching order $H$, and allowing $(B^H+f)\lvert_E$ to hit points. We can safely say that the same method can be used in the case $\dim(E)< Hd$. Namely, we can construct a drift $f:[0,1]\rightarrow \mathbb{R}^d$ which is $(\alpha-\varepsilon)$-Hölder continuous for all small $\varepsilon>0$, with $\alpha:=\dim(E)/d<H$, such that  $(B^H+f)\vert_E$ hits points.

\noindent Now we introduce some useful notations throughout the paper. 

\noindent $|.| $ denotes the usual metric on $\mathbb{R}_+$ and $\lVert \cdot \rVert$ is the Euclidean norm on $\mathbb{R}^d$. For functions $f,g$, $f\asymp g$ means that there exists a constant $\mathsf{c}\geq 1$ such that $\mathsf{c}^{-1}g(\cdot)\leq f(\cdot)\leq \mathsf{c}\,g(\cdot)
$.

\section{Preliminaries on Hausdorff measures and capacities}
In this section, we would like to consider some comparison results that are intended to provide upper bound for the Hausdorff measure (resp. lower bound of the Bessel-Riesz capacity) of a product set $G_1\times G_2$ in terms of the Hausdorff measures (resp. in terms of capacity) of each of the component $G_1$ and $G_2$ with appropriate orders. Such results might be useful in studying the problem of hitting probabilities for $B^H+f$.
First of all, we need to recall definitions of Hausdorff measures as well as Bessel-Riesz capacities in a general metric space. 

Let $(\mathscr{X},\rho)$ be a bounded metric space, $\beta>0$ and $G \subset {\mathscr{X}}$. We define the $\beta$-dimensional Hausdorff measure of $G$ with respect to the metric $\rho$ as

\begin{equation}
	\mathcal{H}_{\rho}^{\beta}(G)=\lim_{\delta \rightarrow 0}\inf \left\{\sum_{n=1}^{\infty}\left(2 r_{n}\right)^{\beta}: G \subseteq \bigcup_{n=1}^{\infty} B_{\rho}\left(r_{n}\right), r_{n} \leqslant \delta\right\},\label{def Haus general metric}
	\end{equation}
	where $B_{\rho}(r)$ denotes an open ball of radius $r$ in the metric space $({\mathscr{X}},\rho)$. 
   \noindent The Hausdorff dimension of $G$ in the metric space $(X,\rho)$ is defined to be 
\begin{equation}\label{def Haus dim}
\dim_{\rho}(G)=\inf\{\beta>0: \mathcal{H}_{\rho}^{\beta}(G)=0\}\, \text{ \,for all\, $G\subset {\mathscr{X}}$}.     
\end{equation}
The usual $\beta$-dimensional Hausdorff measure and the Hausdorff dimension in Euclidean metric are denoted by $\mathcal{H}^{\beta}$ and $\dim(\cdot)$, respectively. Moreover, $\mathcal{H}^{\beta}$ is assumed equal to $1$ whenever $\beta\leq 0$.

We introduce now the Minkowski dimension of $E\subset ({\mathscr{X}},\rho)$. Let $N_{\rho}(E,r)$ be covering number, that is the smallest number of open balls of radius $r$ required to cover $E$. The lower and upper Minkowski dimensions of $E$ are respectively defined as 
	
	\begin{equation*}
	\begin{aligned}
	&\underline{\dim}_{\rho,M}(E):= \liminf_{r\rightarrow 0^+}\frac{\log N_{\rho}(E,r)}{\log(1/r)},\\
	&\overline{\dim}_{\rho,M}(E):= \limsup_{r\rightarrow 0^+}\frac{\log N_{\rho}(E,r)}{\log(1/r)}.
	\end{aligned}       
	\end{equation*}
 
	Equivalently, the upper Minkowski dimension of $E$ can be characterized by 
	\begin{equation}
	\overline{\dim}_{\rho,M}(E)=\inf\{\gamma : \exists\, C\in \mathbb{R}_+\, \text{ such that } N_{\rho}(E,r)\leqslant Cr^{-\gamma} \text{ for all } r>0 \}.\label{upper Minkowski}
	\end{equation}
In the Euclidean case, the upper Minkowski dimension will be denoted by $\overline{\dim}_{M}$.

The Bessel-Riesz type capacity of order $\alpha$ on  the metric space $({\mathscr{X}},\rho)$ is defined by 
	\begin{equation}
	\mathcal{C}_{\rho}^{\alpha}(G)=\left[\inf _{\mu \in \mathcal{P}(G)} \int_{\mathscr{X}} \int_{\mathscr{X}} \varphi_{\alpha}(\rho(u,v)) \mu(d u) \mu(d v)\right]^{-1},\label{def cap general metric}
	\end{equation}
	where $\mathcal{P}(G)$ is the family of probability measures carried by $G$ and the function $\varphi_{\alpha}:(0,\infty)\rightarrow (0,\infty)$ is given by
	\begin{equation}
	\varphi_{\alpha}(r)=\left\{\begin{array}{ll}{r^{-\alpha}} & {\text { if } \alpha>0,} \\ {\log \left(\frac{e}{r \wedge 1}\right)} & {\text { if } \alpha=0,} \\ {1} & {\text { if } \alpha<0.}\end{array}\right.
	\label{radial kernel}
	\end{equation} 
\noindent  We note that for $\alpha <0$ we have $\mathcal{C}_{\rho}^{\alpha }(G)=1$ for any nonempty $G$. The usual Bessel-Riesz capacity of order $\alpha$ in the Euclidean metric will be denoted by $\mathcal{C}^{\alpha}$.

Now let $\left({\mathscr{X}_i},\rho_i\right)$ $i=1,2$ are two metric spaces. Let $\rho_3$ be the metric on $\mathscr{X}_1\times \mathscr{X}_2$ given by
$$
	    \rho_3\left((u,x),(v,y)\right)=\max\{\rho_1(u,v),\rho_2(x,y)\}.
$$
\noindent The following proposition if the main result of this section.
\begin{proposition}\label{1st prop mastrand}
Let $\alpha>0$ and $G_i \subset  \mathscr{X}_i$, $i=1,2$. 
\begin{itemize}
\item[i)] If $G_1$ supports a probability measure $\mu$ that satisfies
\begin{equation}\label{estim frostman}
		\mu\left(B_{\rho_1}\left(a,r\right)\right)\leq \mathbf{\mathfrak{c}}_{1}\,r^{\gamma}\,\,\text{ for all \ensuremath{a\in G_1}, and all \, \ensuremath{0<r\leq diam(G_1)}},
\end{equation}
for some positive constants $\mathbf{\mathfrak{c}}_{1}$ and $\gamma$, then
there exists a constant $\mathbf{\mathfrak{c}}_{2}>0$ such that 
		\begin{equation} \label{lwr cap estim} \mathcal{C}_{\rho_2}^{\alpha-\gamma}(G_{2})\leq \mathbf{\mathfrak{c}}_{2}\,\mathcal{C}_{\rho_3}^{\alpha}(G_1\times G_2).
		\end{equation}
		
\item[ii)]If there exist constants $\gamma'< \alpha$ and $ \mathbf{\mathfrak{c}}_3>0$ such that
\begin{equation}\label{estim cover number 0}
N_{\rho_1}(G_1,r)\leq \mathbb{\mathfrak{c}}_3\, r^{-\gamma'} \text{ for all \, \ensuremath{0<r\leq diam(G_1)}},
\end{equation}
\end{itemize}
then we have
\begin{equation}\label{uppr Hausdorff estim}
     \mathcal{H}_{\rho_3}^{\alpha}(G_1\times G_2)\leq\, \mathbf{\mathfrak{c}}_{4} \mathcal{H}_{\rho_2}^{\alpha-\gamma'}(G_2),
\end{equation}
for some constant $\mathbf{\mathfrak{c}}_4>0$.

Similar estimates hold true if we assume that $G_2$ verifies assumptions $(i)$ and $(ii)$. That is there exist positive constants $\mathbf{\mathfrak{c}}_{5}$ and $\mathbf{\mathfrak{c}}_{6}$ such that 
\begin{equation} \label{lwr cap estim 1} \mathcal{C}_{\rho_1}^{\alpha-\gamma}(G_{1})\leq \mathbf{\mathfrak{c}}_{5}\,\mathcal{C}_{\rho_3}^{\alpha}(G_1\times G_2),
		\end{equation}
		
\begin{equation}\label{uppr Hausdorff estim 1}
     \mathcal{H}_{\rho_3}^{\alpha}(G_1\times G_2)\leq\, \mathbf{\mathfrak{c}}_{6} \mathcal{H}_{\rho_1}^{\alpha-\gamma'}(G_1).
\end{equation}

\end{proposition}

\begin{remark}
1. By the mass distribution principle, we get from assertion i. that $
\dim_{\rho_1}(G_1)\geq \gamma
$.

\noindent 2. Using the characterization of the upper Minkowski dimension \eqref{upper Minkowski} we obtain from assertion ii. that $ \overline{\dim}_{\rho_1,M}(G_1)\leq \gamma'$ and $\mathcal{H}_{\rho_1}^{\gamma'}(G_1)<\infty$. Hence by the Frostman's energy method \cite[Theorem 6.4.6 p.173]{Bishop Peres} we immediately have $\mathcal{C}_{\rho_1}^{\gamma'}(G_1)=0$.
\end{remark}

\noindent The following lemma will help us to establish \eqref{lwr cap estim} (resp. \eqref{lwr cap estim 1}). 
\begin{lemma}\label{estim Frostman lemma}
Let $(\mathscr{X},\rho)$ be a bounded metric space and $\mu$ a probability measure supported on $\mathscr{X}$ satisfying 
\begin{align}\label{Frost cond 2}
    \mu\left(B_{\rho}(u,r)\right) \leq C_1 r^{\kappa} \quad \text{ for all $\,u\in \mathscr{X}$ and $\,r>0$,}
\end{align}
for some positive constants $C_1$ and $\kappa$. Then for any $\theta>0$, there exists $C_2>0$ such that
\begin{align}\label{estim Frostman}
     I(r):=\sup_{v \in \mathscr{X}}\int_\mathscr{X}\frac{\mu(du)}{\left(\max\{\rho(u,v),r\}\right)^{\theta}}\leq C_2\, \varphi_{\theta-\kappa}(r),
\end{align}
for all $r\in (0,\operatorname{diam(\mathscr{X})})$.
\begin{proof} Without loss of generality we can assume that $\operatorname{diam}(\mathscr{X})=1$. 
Three cases need to be discussed here: (i) $\theta<\kappa$, (ii) $\theta = \kappa$ and (iii) $\theta>\kappa$. For the first case $\theta<\kappa$, we have to show only that 
$$
\underset{r\in (0,1)}{\sup}I(r)<\infty.
$$
Indeed, for any $v\in X$ we have
 \begin{align*}\label{decomp energy}
    \int_\mathscr{X} \frac{\mu(du)}{\left(\max\{\rho(u,v),r\}\right)^{\theta}}&\leq 
    \int_\mathscr{X}\frac{\mu(du)}{\rho(u,v)^{\theta}}=\sum_{j=1}^{\infty}\int_{\{u:\rho(u,v)\in (2^{-j},2^{-j+1}]\}}\frac{\mu(du)}{\rho(u,v)^{\theta}}\\
    &\leq \sum_{j=1}^{\infty}2^{j\theta}\mu\left(B_{\rho}(v,2^{-j+1})\right)\leq C_1\,2^{\kappa} \sum_{j=1 }^{\infty}2^{-j(\kappa-\theta)}<\infty.
\end{align*}
  
  \noindent Now for $\theta\geq \kappa$, we have first 
  $$
  I(r)\leq \underset{:=I_1(r)}{\underbrace{\sup_{v\in \mathscr{X}}\int_{\{u: \rho(u,v)< r \}}\frac{\mu(du)}{r^{\theta}}}} + \underset{:=I_2(r)}{\underbrace{\sup_{v\in \mathscr{X}}\int_{\{u:\rho(u,v)\geq r\}}\frac{\mu(du)}{\rho(u,v)^{\theta}}}}
  $$
with $r\in \left(0,1\right)$. By using \eqref{Frost cond 2} we get 
 \begin{equation}\label{I1 power}
     I_1(r)\leq C_1\, r^{\kappa-\theta}.
 \end{equation}
 For estimating $I_2(r)$, we set $j(r):=\inf\{j: 2^{-j}\leq r\}$. Then we have 
 \begin{align}
     \{u:\rho(u,v)\geq r\}\subset \bigcup_{j=1}^{j(r)}\{u:2^{-j}\leq \rho(u,v)<2^{-j+1}\}
 \end{align}
Simple calculations and \eqref{Frost cond 2} ensures that for any $v\in \mathscr{X}$ we have  
 \begin{align}\label{I2 geom}
         \int_{\{u:\rho(u,v)\geq r\}}\frac{\mu(du)}{\rho(u,v)^{\theta}}&\leq \sum_{j=1}^{j(r)}2^{j\,\theta}\mu\left(\{u: 2^{-j}\leq \rho(u,v)<2^{-j+1}\}\right)\nonumber\\
         &\leq C_1\, 2^{\kappa}\sum_{j=1}^{j(r)}2^{j(\theta-\kappa)}.
 \end{align}
 It follows from the definition of $j(r)$ that $2^{-j(r)}\leq r< 2^{-j(r)+1}$. Then, for $\theta=\kappa$ we get easily that 
 \begin{align}\label{I2 log}
     I_1(r)\leq C_1 \quad \text{ and } \quad I_2(r)\leq C_3\,\log(e/r)=C_3\,\varphi_{\theta-\kappa}(r).
 \end{align}
 Hence we get the desired estimation for the case $(ii)$. For the last case $\theta>\kappa$, we use a comparison with a geometric sum  in \eqref{I2 geom} to obtain \begin{align}\label{I2 power}
      I_2(r)\leq C_4\, r^{\kappa-\theta}.
  \end{align} 
Putting \eqref{I1 power} and \eqref{I2 power} all together, the estimation \eqref{estim Frostman} follows.
\end{proof}
\end{lemma}

\begin{proof}[Proof of Proposition \ref{1st prop mastrand}]
We start by proving $(i)$. Let us suppose that $\mathcal{C}_{\rho_2}^{\alpha-\gamma}(G_2)>0$, otherwise there is nothing to prove. It follows that for $\eta \in (0,\mathcal{C}_{\rho_2}^{\alpha-\gamma}(G_2))$ there is a probability measure $m$ supported on $G_2$ such that
\begin{align}
\mathcal{E}_{\rho_2,\alpha-\gamma}(m):=\int_{G_2}\int_{G_2}\varphi_{\alpha-\gamma}(\rho_2(x,y))\,m(dx)m(dy)\leq \eta^{-1}.\label{E,d}
\end{align}
Since $\mu \otimes m$ is a probability measure on $G_1 \times G_2$, then applying Fubini’s theorem and Lemma \ref{estim Frostman lemma} we obtain 
\begin{align}
\begin{aligned}
    \mathcal{E}_{\rho_3,\alpha}(\mu \otimes m)&= \int_{G_1\times G_2} \int_{G_1\times G_2} \frac{\mu\otimes m(dudx) \mu\otimes m(dvdy)}{(\rho_3\left((u,x),(v,y)\right))^{\alpha}}\\
    &\leq  C_2\int_{G_2}\int_{G_2}\varphi_{\alpha-\gamma}(\rho_2(x,y))\,m(dx)m(dy)\leq C_2\,\eta^{-1}.\label{E,d,rho}
\end{aligned}
\end{align}
Consequently we have $\mathcal{C}_{\rho_3}^{\alpha}(G_1\times G_2)\geq  C_2^{-1}\,\eta$. Then we let  $\eta\uparrow\mathcal{C}_{\rho_2}^{\alpha-\gamma}(G_2)$  to conclude that the inequality in \eqref{lwr cap estim} holds true.

Now let us prove $(ii)$. Let $\zeta>\mathcal{H}_{\rho_2}^{\alpha-\gamma'}(G_2)$ be arbitrary. Then there is a covering of $G_2$ by open balls $B_{\rho_2}(x_{n},r_n)$ of radius $r_n$ such that 
	\begin{align}
G_2 \subset \bigcup_{n=1}^{\infty} B_{\rho_2}(x_{n},r_{n}) \quad \text { and } \quad \sum_{n=1}^{\infty} (2r_{n})^{\alpha-\gamma'} \leq \zeta.\label{covering for F}
	\end{align}
For all $n\geq 1$, let $B_{\rho_1}(u_{n,j},r_n)$, $j=1,...,N_{\rho_1}(G_1,r_n)$ be a family of open balls covering  $G_1$.  It follows that the family $B_{\rho_1}(u_{n,j},r_n)\times B_{\rho_2}(x_{n},r_n),\,j=1,...,N_{\rho_1}(G_1,r_n)$, $n \geq  1$ gives a covering of $G_1\times G_2$ by open balls of radius $r_n$ for the metric $\rho_3$. 
		
		It follows from \eqref{estim cover number 0} and \eqref{covering for F} that  
		\begin{align}
		\sum_{n=1}^{\infty} \sum_{j=1}^{N_{\rho_1}(G_1, r_n)}\left(2 r_{n}\right)^{\alpha} \leq \mathbf{\mathfrak{c}}_{3}\,2^{\gamma'}\, \sum_{n=1}^{\infty} (2r_{n})^{\alpha-\gamma'} \leq \mathbf{\mathfrak{c}}_{3}\,\,2^{\gamma'}\, \zeta.\label{covering for E*F}
		\end{align}
		Letting $\zeta\downarrow\mathcal{H}_{\rho_2}^{\alpha-\gamma'}(G_2)$, the inequality in \eqref{uppr Hausdorff estim} follows with $\mathbf{\mathfrak{c}}_{4}=\mathbf{\mathfrak{c}}_{3}\,\,2^{\gamma'}$.
\end{proof}
 In the following we give a sufficient condition ensuring hypotheses $(i)$ and $(ii)$ of Proposition \ref{1st prop mastrand}.
\begin{proposition}\label{2nd prop mastrand} The following condition 
$$
0<\gamma<\dim_{\rho_1}(G_1)\leq \overline{\dim}_{\rho_1,M}(G_1)<\gamma'<\alpha,
$$	
is sufficient to achieve \eqref{estim frostman} and \eqref{estim cover number 0}.
\begin{proof}
$(i)$ (resp. $(ii)$) is a direct consequence of Frostman's Theorem (resp. the characterization \eqref{upper Minkowski}) 
\end{proof}
\end{proposition}

It is well known that Hausdorff and Minkowski dimensions agree for many Borel sets $E$. Often this is linked on the one hand to the geometric properties of the set, on the other hand it is a consequence of the existence of a sufficiently regular measure. Among the best known are Ahlfors-David regular sets defined as follows.
	
\begin{definition}
Let $(\mathscr{X},\rho)$ be a bounded metric space, $\gamma>0$ and $G\subset \mathscr{X}$. We say that $G$ is $\gamma$-Ahlfors-David regular if there exists a finite positive Borel measure $\mu$ supported on $G$ and positive constant $\mathbf{\mathfrak{c}}_{\gamma}$ such that 
\begin{equation}\label{condition S}
		\mathbf{\mathfrak{c}}_{\gamma}^{-1}\,r^{\gamma}\leq\mu\left(B_{\rho}\left(a,r\right)\right)\leq\mathbf{\mathfrak{c}}_{\gamma}\,r^{\gamma}\,\,\text{ for all \ensuremath{a\in G}, and all \, \ensuremath{0<r\leq1}}.
\end{equation}
\end{definition}
\noindent For a Borel set $E\subset \mathbb{R}^n$ satisfying the condition \eqref{condition S} with $\rho$ is the euclidean metric of $\mathbb{R}^n$, it is shown in \cite[Theorem 5.7 p.$80$]{Mattila} that
$$ \gamma=\dim(E)=\underline{\dim}_M(E)=\overline{\dim}_M(E).
$$
This statement still true in a general metric space $(\mathscr{X},\rho)$, it suffices to go through the same lines of the proof of the euclidean case. Here we provide some examples of such sets.

\begin{example}
\textbf{i)} If $E$ is the whole interval $I$ then the condition \eqref{condition S} is satisfied with $\gamma=1$. This leads to the conclusion that the measure $\mu$ can be chosen as the normalized Lebesgue measure on $I$.\\
\textbf{ii)} The Cantor set $C(\lambda)$, $0<\lambda<1/2$, subset of $I$ with $\mu$ is the $\gamma$-dimensional Hausdorff measure restricted to $C(\lambda)$ where $\gamma=\dim C(\lambda)=\log(2)/\log(1/\lambda)$. For more details see \cite[Theorem 4.14 p.$67$]{Mattila}. In general, self similar subsets of $\mathbb{R}$ satisfying the open set condition are standard examples of Ahlfors-David regular sets, see \cite{Hutchinson}.
\end{example}	

 The following proposition states that when $G_1$ (resp. $G_2$) is $\gamma_1$-Ahlfors-David regular set (resp. $\gamma_2$-Ahlfors-David regular set), then inequalities \eqref{lwr cap estim} and \eqref{uppr Hausdorff estim} of Proposition \ref{1st prop mastrand} are checked for $\gamma=\gamma'=\gamma_1$ (resp. for $\gamma=\gamma'=\gamma_2$).

\begin{proposition}\label{3rd prop mastrand} Let $\alpha>0$, and
assume that $G_1$ is $\gamma_1$-Ahlfors-David regular set.

\begin{itemize}
    \item[i)] If $\gamma_1\leq \alpha$ then  
    \begin{equation}\label{ineq cap 1}
        \mathcal{C}_{\rho_2}^{\alpha-\gamma_1}(G_{2})\leq \mathbf{\mathfrak{c}}_{7}\,\mathcal{C}_{\rho_3}^{\alpha}(G_1\times G_2).
    \end{equation}
    \item[ii)] If $\gamma_1<\alpha$ then 
    \begin{equation}\label{ineq haus 1}
  \mathcal{H}_{\rho_3}^{\alpha}(G_1\times G_2)\leq\, \mathbf{\mathfrak{c}}_{8} \mathcal{H}_{\rho_2}^{\alpha-\gamma_1}(G_2).
\end{equation}
\end{itemize}

\noindent Similar estimates hold true under the assumption $G_2$ is $\gamma_2$-Ahlfors-David regular set. Precisely we have
\begin{equation}\label{lwr cap estim 2}
        \mathcal{C}_{\rho_1}^{\alpha-\gamma_2}(G_{1})\leq \mathbf{\mathfrak{c}}_{9}\,\mathcal{C}_{\rho_3}^{\alpha}(G_1\times G_2),
    \end{equation}
    
    \begin{equation}\label{uppr Hausdorff estim 2}
  \mathcal{H}_{\rho_3}^{\alpha}(G_1\times G_2)\leq\, \mathbf{\mathfrak{c}}_{10} \mathcal{H}_{\rho_1}^{\alpha-\gamma_2}(G_1).
\end{equation}
\begin{proof}
In order to prove \eqref{ineq cap 1} and \eqref{ineq haus 1} it is sufficient to check that conditions \eqref{estim frostman} and \eqref{estim cover number 0} are satisfied with $\gamma=\gamma'=\gamma_1$. Firstly, \eqref{estim frostman} is no other than the right inequality in \eqref{condition S}. On the other hand, let $0<r\leq 1$ and $P_{\rho_1}(G_1,r)$ be the packing number, that is the greatest number of disjoint balls $B_{\rho_1}(x_j,r)$ with $x_{j}\in G_1$. The left inequality of \eqref{condition S} ensures that 
$$
\mathbf{\mathfrak{c}}_{\gamma_1}^{-1}\, P_{\rho_1}(G_1,r)\, r^{\gamma_1}\leq \sum_{j=1}^{P_{\rho_1}(G_1,r)}\mu\left(B_{\rho_1}(x_j,r)\right)=\mu (G_1)\leq 1.
$$ 
Using the well known fact that
$N_{\rho_1}(G_1,2\,r)\leq P_{\rho_1}(G_1,r)$,
we obtain the desired estimation \eqref{estim cover number 0}. 
\end{proof}
\end{proposition}
\begin{remark}
Notice that when both $G_i$, $i=1,2$ are $\gamma_i$-Ahlfors David regular sets for some constant $\gamma_i>0$, then there exist two positive constants $\mathbf{\mathfrak{c}}_{11}$ and $\mathbf{\mathfrak{c}}_{12}$ such that 
\begin{align}
\mathcal{C}_{\rho_3}^{\gamma_1+\gamma_2}(G_1\times G_2)\geq \mathbf{\mathfrak{c}}_{11}\,\mathcal{C}_{\rho_2}^{\gamma_2}(G_{2})\vee \mathcal{C}_{\rho_1}^{\gamma_1}(G_{1}) \quad \text{ and }\quad \mathcal{H}_{\rho_3}^{\gamma_1+\gamma_2}(G_1\times G_2)\leq \mathbf{\mathfrak{c}}_{12}\,\mathcal{H}_{\rho_2}^{\gamma_2}(G_{2})\wedge \mathcal{H}_{\rho_1}^{\gamma_1}(G_{1}). 
\end{align}
\end{remark}
\section{Hitting probabilities for fractional Brownian motion with deterministic regular drift}
\noindent Let $H \in (0,1)$ and $B^{H}_0=\left\lbrace B^{H}_0(t),t \geq 0\right\rbrace $ be a real-valued fractional Brownian motion  of Hurst index $H$ defined on a complete  probability space $(\Omega,\mathcal{F},\mathbb{P})$, i.e. a real valued Gaussian process with stationary increments and covariance function given by $$\mathbb{E}(B^{H}_0(s)B^{H}_0(t))=\frac{1}{2}(|t|^{2H}+|s|^{2H}-|t-s|^{2H}).$$
	Let  $B^{H}_1,...,B^{H}_d$ be $d$ independent copies of $B^{H}_0$. The stochastic process $B^{H}=\left\lbrace B^{H}(t),t\geq 0\right\rbrace $  given by  
	\[
	B^{H}(t)=(B^{H}_1(t),....,B^{H}_d(t)) ,
	\]
	is called a $d$-dimensional fractional Brownian motion of Hurst index $H \in (0,1)$.

For $\mathbf{d}$ a metric on $\mathbb{R}_+$, we define on $\mathbb{R}_+\times\mathbb{R}^d$ the metric $\rho_{\mathbf{d}}$ as follows 	
\begin{equation}
\rho_{\mathbf{d}}((s,x),(t,y))=\max\{\mathbf{d}\left(t,s\right),\|x-y\|\}\quad \forall (s,x), (t,y)\in \mathbb{R}_+\times\mathbb{R}^d.\label{def product metric}
\end{equation}
\noindent We denote by $\mathbf{d}_{H}$ the canonical metric of $B^H$ given by
\begin{align}\label{canonical metric of fBm}
      \mathbf{d}_{H}(s,t):=|t-s|^H \, \quad \text{for all } s,t \in \mathbb{R}_+.
\end{align}
The associated metric $\rho_{\mathbf{d}_{H}}$ on $\mathbb{R}_+\times \mathbb{R}^d$ is called the parabolic metric.

Let $\mathrm{\mathbf{I}} = [a, b]$, where $a<b \in (0,1]$ are fixed constants. For $\alpha\in (0,1)$, $C^{\alpha}(\mathbf{I})$ is the space of $\alpha$-Hölder continuous function $f : [a,b] \longrightarrow \mathbb{R}^d $ equipped with the norm 
	\begin{equation}\label{mu-norm}
		\lVert f \rVert_{\alpha}:= \sup_{s\in \mathbf{I}} \lVert f(s) \rVert +\underset{\substack{s, t \in \mathbf{I} \\ s \neq t}}{\sup}\dfrac{\lVert f(t)-f(s) \rVert}{|t-s|^{\alpha}}<+\infty.
	\end{equation}
First we state the following result, which is easily deduced from \cite[Theorem 2.6]{Chen2013}.

\begin{theorem}\label{main theorem hitting}
Let $\{B^H(t), t\in [0,1]\}$ be a $d$-dimensional fractional Brownian motion and $f \in C^{H}\left(\mathbf{I}\right)$. Let $F\subseteq \mathbb{R}^d$ and $E\subset \mathbf{I}$ are two compact subsets. Then there exist a constant $c\geq 1$ depending on $\mathbf{I}$, $F$, $H$ and $f$ only, such that
\begin{equation}
c^{-1}\mathcal{C}_{\rho_{\mathbf{d}_{H}}}^{d}(E\times F)\leq \mathbb{P}\{(B^H+f)(E)\cap F\neq \emptyset \}\leq c\,\,\mathcal{H}_{\rho_{\mathbf{d}_{H}}}^{d}(E\times F).\label{upper-lower bounds 1}
\end{equation}
	\end{theorem}
	The goals of this section is to provide the hitting probabilities estimates for some particular sets $E$ and $F$. Such estimates would be helpful in the next section to establish a kind of of sharpness of the $H$-Hölder regularity of the drift $f$ in Theorem \ref{main theorem hitting}.
 
\subsection{Hitting probabilities estimates when $E$ is $\beta$-Ahlfors-David regular set}


	First let us recall that, when $E$ is an interval, Corollary 2.2 in \cite{Chen Xiao 2012} ensures that there exists a constant  $ \mathsf{c}\geq 1$ depending only on $E$, $F$ and $H$ such that 
	$$\mathsf{c}^{-1}\,\mathcal{C}^{d-1/H}(F)\leq \mathbb{P}\left\{ B^H(E)\cap F\neq \emptyset \right\}\leq \mathsf{c} \,\mathcal{H}^{d-1/H}(F),$$ for any Borel set $F\subseteq \mathbb{R}^d$. 
	Our next goal is to establish such estimates for $(B^H+f)$ and $E$ a Borel set.

	\begin{proposition}\label{corollary hitting}
		Let $B^H$, $f$, $E$ and $F$ as in Theorem \ref{main theorem hitting} such that $0<\dim(E) \leq \overline{\dim}_M(E)<Hd$. Then for any  $0<\beta<\dim(E)\leq \overline{\dim}_M(E)<\beta'<Hd$, we have 
		\begin{equation}
		\mathtt{c}_1^{-1}\,\mathcal{C}^{d-\beta/H}(F)\leq \mathbb{P}\{(B^H+f)(E)\cap F\neq \emptyset \}\leq \mathtt{c}_1\,\mathcal{H}^{d-\beta'/H}(F) ,\label{lower bound 2}
		\end{equation} 
		where $\mathtt{c}_{1}$ is a positive constant depends only on $E$, $F$, $H$, $\mathrm{K}_f$, $\beta$ and $\beta'$. 
\end{proposition}
\begin{proof}
We take $X_1=[0,1]$, $\rho_1(s,t):=|t-s|^H$ and $X_2=\mathbb{R}^d$, $\rho_2(x,y)=\|x-y\|$. First, let us note that $\dim_{\rho_1}(E)=\dim(E)/H$ and $\overline{\dim}_{\rho_1,M}(E)=\overline{\dim}_M(E)/H$. Then \eqref{lower bound 2} follows from Theorem \ref{main theorem hitting}, \eqref{lwr cap estim} and \eqref{uppr Hausdorff estim} of Proposition \ref{1st prop mastrand}, and Proposition \ref{2nd prop mastrand}.
\end{proof}

	As we will see, the above hitting probabilities estimates become more accurate when $E$ is $\beta$-Ahlfors-David regular sets.
	\begin{proposition}\label{corollary hitting 2}
		Let $B^H$, $f$, $E$ and $F$ as in Theorem \ref{main theorem hitting} with ${E\subset \left(\mathbf{I},|.|  \right)} $ is a $\beta$-Ahlfors-David regular for some $\beta\in (0,1]$. Then there is a positive constant $\mathtt{c}_{2}$ which depends on $E$, $F$, $H$, $K_f$ and $\beta$, such that 
		\begin{equation}
		\mathtt{c}_{2}^{-1}\, \mathcal{C}^{d-\beta/H}(F)\leq \mathbb{P}\{(B^H+f)(E)\cap F\neq \emptyset \}\leq \mathtt{c}_{2}\,\mathcal{H}^{d-\beta/H}(F).\label{upper-lower bound 2}
		\end{equation}
	\end{proposition}
\begin{proof}
Three cases are to be discussed here: (j) $\beta <Hd$, (jj) $\beta =Hd$ and (jjj) $\beta >Hd$. Let us point out first that for the lower bound, the interesting cases are (j) and (jj) while for the upper bound it is the case (j) which requires proof. Using Proposition \ref{3rd prop mastrand} with $(X_1,\rho_1)$ and $(X_2,\rho_2)$, are taken as in the proof of Proposition \ref{corollary hitting}, and $\alpha=d$ we get the first case (j). For the case (jj) we only use Proposition \ref{3rd prop mastrand} $(i)$.
\end{proof}
\noindent Following the same pattern as above we get the following corollary. 
	\begin{corollary}\label{corollary hitting CX12}
Let $B^H$, $E$ and $F$ as in Proposition \ref{corollary hitting 2}. Then there is a positive constant $\mathtt{c}_{3}$ which depends on $E$, $F$, $H$ and $\beta$, such that 
\begin{equation}\label{upper-lower boundCX12}
\mathtt{c}_{3}^{-1}\, \mathcal{C}^{d-\beta/H}(F)\leq \mathbb{P}\{B^H(E)\cap F\neq \emptyset \}\leq 		\mathtt{c}_{3}\,\mathcal{H}^{d-\beta/H}(F).
\end{equation}
\end{corollary}
\begin{remark}
We note that \cite[Corollary 2.2]{Chen Xiao 2012} corresponds to the particular case $E=\mathbf{I}$ for which $\beta=1$.
	\end{remark}
	
\subsection{Hitting probabilities estimates when $F$ is $\gamma$-Ahlfors-David regular}
	Now we get parallel results to those given in Proposition \eqref{corollary hitting} and \eqref{corollary hitting 2}, emphasizing regularity properties of $F$ instead of $E$. First, we have the following result.
	\begin{proposition}\label{cor hitting 3}
	    Let $B^H$, $f$, $E$ and $F$ as in Theorem \ref{main theorem hitting}, such that $0<\dim(F)\leq \overline{\dim}_M(F)<d$. Then for all $\gamma$ and $\gamma'$ such that $0<\gamma<\dim(F)\leq \overline{\dim}_M(F)<\gamma'<d$, we have
\begin{equation}
\mathtt{c}_{4}^{-1}\, \mathcal{C}^{H(d-\gamma)}(E)\leq \mathbb{P}\{(B^H+f)(E)\cap F\neq \emptyset \}\leq \mathtt{c}_{4}\,\mathcal{H}^{H(d-\gamma')}(E),\label{upper-lower bound 2}
\end{equation}
where $\mathtt{c}_{4}$ is a positive constant which depends on $E$, $F$, $H$, $\mathrm{K}_f$, $\gamma$ and $\gamma'$.
\begin{proof}
Once again taking $(X_1,\rho_1)$ and $(X_2,\rho_2)$ as in the proof of Proposition \ref{corollary hitting} and $\alpha=d$, it is easy to see that $\mathcal{C}_{\rho_1}^{d-\gamma}(\cdot)\equiv \mathcal{C}^{H(d-\gamma)}(\cdot)$ and $\mathcal{H}_{\rho_1}^{d-\gamma'}(\cdot) \equiv \mathcal{H}^{H(d-\gamma')}(\cdot)$. Hence \eqref{upper-lower bound 2} is a consequence of Theorem \ref{main theorem hitting}, \eqref{lwr cap estim 1} and \eqref{uppr Hausdorff estim 1} of Proposition \ref{1st prop mastrand}, and Proposition \ref{2nd prop mastrand}.
	\end{proof}
	\end{proposition}
	
When ${F\subset \left( \mathbb{R}^d,\lVert \cdot \rVert\right) } $ is a $\gamma$-Ahlfors-David regular, we have the following result, which could be proven similarly to Proposition \ref{corollary hitting 2} by making use of \eqref{lwr cap estim 2} and \eqref{uppr Hausdorff estim 2}. 

\begin{proposition}\label{cor hitting 4}
Let $B^H$, $f$, $E$ and $F$ as in Theorem \ref{main theorem hitting}, such that $F$ is a $\gamma$-Ahlfors-David regular compact subset of $[-M,M]^d$ for some $\gamma \in (0,d]$. Then there is a positive constant $\mathtt{c}_{5}$ which depends on $E$, $F$, $H$, $\mathrm{K}_f$ and $\gamma$ only, such that 
\begin{equation}
\mathtt{c}_{5}^{-1}\, \mathcal{C}^{H(d-\gamma)}(E)\leq \mathbb{P}\{(B^H+f)(E)\cap F\neq \emptyset \}\leq \mathtt{c}_{5}\,\mathcal{H}^{H(d-\gamma)}(E).\label{hitting F ahlfors}
\end{equation}
\end{proposition}
	
\begin{corollary}
Let $B^H$, $E$ and $F$ as in Proposition \ref{cor hitting 4}. Then there is a positive and constant $\mathtt{c}_{6}$ which depends on $E$, $F$, $H$ and $\gamma$, such that 
\begin{equation}\label{upper-lower boundCX12}
\mathtt{c}_{6}^{-1}\, \mathcal{C}^{H(d-\gamma)}(E)\leq \mathbb{P}\{B^H(E)\cap F\neq \emptyset \}\leq 		\mathtt{c}_{6}\,\mathcal{H}^{H(d-\gamma)}(E).
\end{equation}
\end{corollary}
    
    \section{Sharpness of the Hölder regularity of the drift:}

	This subsection brings to light {the essential role} of the $H$-Hölder regularity assumed on the drift $f$ in the following sens: The result of Theorem \ref{main theorem hitting} fails to hold when the deterministic drift $f$ has a modulus of continuity $\mathsf{w}(\cdot)$ {satisfying $$r^H=o\left(\mathsf{w}(r)\right) \,\,\text{and}\,\, \mathsf{w}(r)=o\left(r^{H-\iota}\right) \,\,\text{when}\,\, r\rightarrow 0 \,\,\text{for all}\,\, \iota\in (0,H).$$ In this respect, we have to introduce some tools allowing us to reach our target.}

Let $\mathscr{L}$ be the class of all continuous functions $\mathsf{w}:[0, 1] \rightarrow(0, \infty)$, $\mathsf{w}(0)=0$, which are increasing on some interval $(0, r_0]$ with $r_0=r_0(\mathsf{w})\in (0,1)$. Let $\mathsf{w}\in \mathscr{L}$ be fixed. A continuous function $f$ is said to belong to the space $C^{\mathsf{w}}(\mathbf{I})$ if and only if
$$
\sup _{\substack{s, t \in \mathbf{I} \\ s \neq t}} \frac{|f(s)-f(t)|}{\mathsf{w}(|s-t|)}<\infty .
$$
It is obvious that the space $C^{\mathsf{w}}(\mathbf{I})$ is a Banach space with the norm
$$
\|f\|_{\mathsf{w}}=\sup_{s \in \mathbf{I}}|f(s)|+\sup _{\substack{s, t \in \mathbf{I} \\ s \neq t}} \frac{|f(s)-f(t)|}{\mathsf{w}(|s-t|)}.
$$
For $\alpha\in (0,1)$ and $\mathsf{w}(t)=t^\alpha$, $C^{\mathsf{w}}(\mathbf{I})$ is nothing but the usual space $C^{\alpha}(\mathbf{I})$.

Let $x_0\in(0,1]$ and $l:(0, x_0]\rightarrow \mathbb{R}_+$ be a slowly varying function at zero in the sens of Karamata (cf. \cite{Bingham et al}). It is well known that $l$ has the representation 
		\begin{align}\label{rep slow var}
        l(x)=\exp\left(\eta(x) + \int_{x}^{x_0}\frac{\varepsilon(t)}{t}dt\right),
    \end{align}
    where $\eta,\varepsilon: [0,x_0)\rightarrow \mathbb{R}$ are Borel measurable and bounded functions such that 
    $$ 
    \lim_{x\rightarrow 0}\eta(x)=\eta_0\in (0,\infty) \quad \text{ and  } \quad \lim_{x\rightarrow 0}\varepsilon(x)=0.
    $$
An interesting property of slowly varying functions which gives some intuitive meaning to the notion of "slow variation" is that for any $\tau>0$ we have 
    	\begin{equation}\label{slow var property}
    		x^{\tau}\, l(x) \rightarrow 0 \quad \text{ and } x^{-\tau}\,l(x)\rightarrow \infty \quad \text{ as } x\rightarrow 0.
    \end{equation}
It is known from Theorem 1.3.3 and Proposition 1.3.4 in \cite{Bingham et al} and the ensuing discussion that there exists {a function $\mathcal{C}^{\infty}$ near zero} $\tilde{l}:(0,x_0]\rightarrow \R_+$  such that  $l(x)\thicksim \tilde{l}(x)$ when $x\rightarrow 0$, and $\tilde{l}(\cdot)$ has the following form
    \begin{equation}\label{nice rep slow var}
        \tilde{l}(x)=\mathsf{c}\,\exp\left( \int_{x}^{x_0}\frac{\tilde{\varepsilon}(t)}{t}dt\right),
    \end{equation}
    for some positive constant $\mathsf{c}$ and {$\tilde{\varepsilon}(x) \rightarrow 0$}. Such function is called normalized slowly varying function (Kohlbecker 1958), and in this case 
    \begin{equation}\label{epsilon calcul}
     \tilde{\varepsilon}(x)=-x\, \tilde{\ell}^{\prime}(x)/\tilde{\ell}(x) \quad \text{for all }  \,x\in (0,x_0).
    \end{equation} 
    A function $\mathsf{v}_{\alpha,{\ell}}:[0, x_0]\rightarrow \mathbb{R}_+$ is called regularly varying function at zero with index {$\alpha \in (0,1)$} if and only if there exists a slowly varying function ${\ell}$, called {the slowly varying part of $\mathsf{v}_{\alpha,\ell}$}, such that
    \begin{align}\label{norm reg var function}
    \mathsf{v}_{\alpha,{\ell}}(0)=0 \quad \text{ and }  \quad \mathsf{v}_{\alpha,{\ell}}(x)=x^{\alpha}\,{\ell}(x),\quad x\in (0,x_0).
    \end{align}
   $\mathsf{v}_{\alpha,\ell}$ is called a normalized regularly varying if its slowly varying part is normalized slowly varying at zero.
   In the rest of this work, since the value of $x_0$ is unimportant because $\tilde{\ell}(x)$ and $\tilde{\varepsilon}(x)$ may be altered at will for $x\in (x_0,1]$, one can choose $x_0=1$ without loss of generality. Furthermore, we will only consider normalized {regularly/slowly} varying function.

Here are some interesting {properties} of normalized {regularly} varying functions. {Let $\mathsf{v}_{\alpha,\ell}$ be a normalized regularly varying at zero with normalized slowly varying part $l$. \begin{lemma}\label{req on slow var}
1. There exists small enough $x_1>0$ such that
$$
 \lim_{x\downarrow 0}\mathsf{v}_{\alpha,\ell}^{\prime}(x)=+\infty\quad \text{ and }\quad \mathsf{v}_{\alpha,\ell}(\cdot) \,\,\text{ is increasing on $(0,x_1]$}.
 $$  
    2. If in addition we assume that $l$ is a  $\mathcal{C}^{2}$ function such that 
   \begin{equation}\label{concav cond}
   	\liminf_{x\downarrow 0}x\,\varepsilon^{\prime}(x)=0,
   \end{equation}
     where $\varepsilon$ is given by \eqref{epsilon calcul}. Then there exists small enough $x_2>0$ such that $\mathsf{v}_{\alpha,\ell}$ is increasingly concave on $(0,x_2]$. Moreover for all $x_3\in (0,x_2)$ and all $\mathsf{c}>0$ small enough there exists $r_0 <x_3$ such that
         $$
        \mathsf{v}_{\alpha,\ell}(t)-\mathsf{v}_{\alpha,\ell}(s)\leq \textbf{{c}}\, \mathsf{v}_{\alpha,\ell}(t-s) \quad \text{ for all $s,t\in [x_3,x_2]$\,\, such that $0<t-s<r_0$}.
        $$       
    \end{lemma}
    \begin{proof}
        1. It stemmed from  
        $$ 
        \mathsf{v}_{\alpha,\ell}^{\prime}(x)=x^{\alpha-1}\,\ell(x)\,\left(\alpha-\varepsilon(x)\right),
        $$ and \eqref{slow var property}.
    
    2.    It is easy to check that
        
$$
     \mathsf{v}_{\alpha,\ell}^{\prime\prime}(x)= x^{\alpha-2}\, \ell(x)\left[ \left(\alpha-1-\varepsilon(x)\right)\left(\alpha-\varepsilon(x)\right)-x\,\varepsilon^{\prime}(x)\right].
$$
\eqref{slow var property} and hypothesis \eqref{concav cond} ensure that $\lim_{x\downarrow 0}\mathsf{v}_{\alpha,\ell}^{\prime\prime}(x)=-\infty$.
Then there exists $x_2>0$ small enough such that $\mathsf{v}_{\alpha,\ell}^{\prime}(x)>0$ and  $\mathsf{v}_{\alpha,\ell}^{\prime\prime}(x)<0$ for all $x\in (0,x_2]$. Thus $\mathsf{v}_{\alpha,\ell}$ is increasingly concave on $(0,x_2]$.

For the rest, let $x_3 \in (0,x_2]$ and $\mathsf{c}>0$ be arbitrary. Let $s<t\in (x_3,x_2)$ and $r< x_3$, then using the monotonicity of $\mathsf{v}_{\alpha,\ell}^{\prime}$ we have for $0<t-s<r$ that
\begin{align}
\frac{\mathsf{v}_{\alpha,\ell}(t)-\mathsf{v}_{\alpha,\ell}(s)}{\mathsf{v}_{\alpha,\ell}(t-s)}
            \leq \frac{\mathsf{v}_{\alpha,\ell}^{\prime}(x_3)}{\mathsf{v}_{\alpha,\ell}^{\prime}(r)}.
        \end{align}
Since $\lim_{x\downarrow 0}\mathsf{v}_{\alpha,\ell}^{\prime}(x)=+\infty$, we can choose $r_0$ to be smallest $r$ guaranteeing that the term $\mathsf{v}_{\alpha,\ell}^{\prime}(x_3)/\mathsf{v}_{\alpha,\ell}^{\prime}(r)$ will be smaller than $\textbf{c}$. This finishes the proof.
        \end{proof}
\begin{remark}
As a consequence of the Lemma \ref{req on slow var}, for any normalized regularly varying at zero $\mathsf{v}_{\alpha,\ell}$ that checks the condition \eqref{concav cond},  $(s,t)\mapsto \mathsf{v}_{\alpha,\ell}(|t-s|)$ defines a metric on {$[a,x_2]^2$}. 
\end{remark}
}
Let $\ell$ be a normalized slowly varying function at zero. Now we consider the continuous function $\mathsf{w}_{H,l}$ given by
\begin{align}\label{modulo of our drift}
 \mathsf{w}_{H,\ell}(0)=0 \quad  \text{ and } \quad \mathsf{w}_{H,\ell}(x)= x^H\,\ell(x)\, \log^{1/2}(1/x), \quad x\in (0,1].
\end{align}
It is easy to see that $\ell(x)\, \log^{1/2}(1/x)$ stills a normalized slowly varying satisfying \eqref{nice rep slow var} with $\bar{\varepsilon}(\cdot)=\varepsilon(\cdot)-\log^{-1}(1/\cdot)/2$. {Hence assertion 1 of Lemma \ref{req on slow var}
 provides that $\mathsf{w}_{H,\ell}$ is increasing on some interval $(0,x_1]$ with $x_1\in (0,1)$}. Therefore $\mathsf{w}_{H,\ell}\in \mathscr{L}$.

\noindent  If we assume in addition that $\ell$ satisfies
\begin{equation}\label{cond modulus cont}
	\underset{x\rightarrow 0}{\limsup}\,\,
\ell(x)\log^{1/2}(1/x)=+\infty,
\end{equation}  
then the following inclusions hold 
\begin{align}\label{compar spaces}
C^{H}(\mathbf{I})\subsetneq C^{\mathsf{w}_{H,\ell}}(\mathbf{I})\subsetneq \underset{\tau>0}{\bigcap} C^{H-\tau}(\mathbf{I}).
\end{align}
{
Let $\theta_H$ be the normalized regularly varying function defined in \eqref{inv spectr density rep} with a normalized slowly varying part that satisfies
	\begin{equation}\label{cond on L_theta}
		\left\{ \begin{array}{l}
			\textbf{i.}\quad \underset{x\rightarrow+\infty}{\liminf}\,\,L_{\theta_{H}}(x)=0,\quad \underset{x\rightarrow+\infty}{\limsup}\,\,L_{\theta_{H}}(x)<+\infty\\
			\\
            \textbf{ii.}\,\,\, \underset{x\rightarrow +\infty}{\limsup} \,\,x\, \varepsilon^{\prime}_{\theta}(x)=0.
		\end{array}\right.
	\end{equation}
Here are some examples of normalized slowly varying functions for which the above conditions are satisfied
\[
 L_{\theta_{H}}(x)=\log^{-{\beta}}\left(x\right), \quad 
L_{\theta_{H}}(x)=\exp\left(-\log^{\alpha}\left(x\right)\right), \quad \alpha \in (0,1) {\text{ and } \beta>0}.
\]
In what follows, we will adopt the following notation}
\begin{align}\label{l_theta}
{ \ell_{\theta_H}(\cdot):=\mathsf{c}_H^{1/2}\,L^{-1/2}_{\theta_{H}}(1/\cdot)}.
\end{align}
Now let us give the main result of this section. 
\begin{theorem}\label{theo sharpness}
Let $\{B^H(t), t\in [0,1]\}$ be a $d$-dimensional fractional Brownian motion. Then there exist a function $f\in C^{\mathsf{w}_{H,{\ell_{\theta_H}}}}(\mathbf{I})\setminus C^{H}\left(\mathbf{I}\right)$, 
and compact sets $E\subset \mathbf{I}$ and $F\subset \mathbb{R}^d$ such that
	    \begin{equation}\label{contre exple}
	    \mathcal{C}_{\rho_{\mathbf{d}_{H}}}^{d}\left(E\times F\right)=\mathcal{H}_{\rho_{\mathbf{d}_{H}}}^d\left(E\times F\right)=0 \quad \text{ and } \quad \mathbb{P}\left\{(B^H+f)(E)\cap F\neq \varnothing \right\}>0.
	    \end{equation}
	    In other words \eqref{upper-lower bounds 1} fails to hold.
	\end{theorem}
\begin{remark}
\textbf{i)} It is worthwhile mentioning that $C^{\mathsf{w}_{H,\ell_{\theta_H}}}(\mathbf{I})$ verifies \eqref{compar spaces} as \,${\ell_{\theta_H}(\cdot)}$ defined in \eqref{l_theta} meets the condition  \eqref{cond modulus cont} via the first term in assertion \textbf{i.} of \eqref{cond on L_theta}, i.e. $\underset{x\rightarrow+\infty}{\liminf}\,\,L_{\theta_{H}}(x)=0$.

\noindent \textbf{ii)} It follows from the fact \eqref{compar spaces} that the drift $f$ in Theorem \ref{theo sharpness} belongs to $\underset{{\tau>0}}{\bigcap}\,C^{H-\tau}\left(\mathbf{I}\right)\setminus C^{H}\left(\mathbf{I}\right)$.
\end{remark}

Before drawing up the proof we provide the tools to be used. Let $\delta_{\theta_H}$ be the function given by the representation \eqref{increments variance}.  Theorem 7.3.1 in \cite{Marcus-Rosen} tells us that $\delta_{\theta_H}$ is normalised regularly varying with index $H$ with a slowly varying part {$\ell_{\delta_{\theta_H}}$} that satisfies
	\begin{equation}\label{slow var lim}
		\ell_{\theta_H}(h)\thicksim \ell_{\delta_{\theta_H}}(h)\quad \text{as $h\rightarrow 0$}.
\end{equation}
For more details see \eqref{ppty of slow var part}. Now we consider another probability space $(\Omega^{\prime},\mathcal{F}^{\prime},\mathbb{P}^{\prime})$ on which we define the real valued centered Gaussian process with stationary increments $B_0^{\delta_{\theta_H}}$, satisfying $B_0^{\delta_{\theta_H}}(0)=0$ a.s. and 
\begin{equation}\label{delta compar}
	\mathbb{E}^{\prime}\left(B_0^{\delta_{\theta_H}}(t)-B_0^{\delta_{\theta_H}}(s)\right)^2=\,\delta_{\theta_H}^2(|t-s|)\quad \text{ for all $t,s\in [0,1]$}.
\end{equation}
{Proposition \ref{Construc of stat increm process} gives a way to construct this process and Theorem \ref{modulus of continuity} provides us its modulus of continuity, that is $B_0^{\delta_{\theta_H}}\in C^{\mathsf{w}_{H,{\ell_{\theta_H}}}}(\mathbf{I})$ $\mathbb{P}'$-almost surely. The $d$-dimensional version of the process $B^{\delta_{\theta_H}}_0$ is the process  $B^{\delta_{\theta_H}}(t):=(B^{\delta_{\theta_H}}_1(t),....,B^{\delta_{\theta_H}}_d(t))$, where $B^{\delta_{\theta_H}}_1,...,B^{\delta_{\theta_H}}_d$ are $d$ independent copies of $B^{\delta_{\theta_H}}_0$. Let $Z$ be the $d$-dimensional mixed process defined on the product space $(\Omega\times\Omega^{\prime},\mathcal{F}\times\mathcal{F}^{\prime},\mathbb{P}\otimes\mathbb{P}^{\prime})$ by 
\begin{align}
	Z(t,(\omega,\omega^{\prime}))=B^H(t,\omega)+B^{\delta_{\theta_H}}(t,\omega^{\prime}) \text{ for all  } t\in  [0,1] \text{ and } (\omega,\omega^{\prime})\in \Omega\times\Omega^{\prime}.
 \end{align}
It is easy to see that the components of  $Z=(Z_1,...,Z_d)$ are independent copies of a real valued Gaussian process $Z_0= B_0^H+ B^{\delta_{\theta_H}}_0$ on $(\Omega\times\Omega^{\prime},\mathcal{F}\otimes\mathcal{F}^{\prime},\mathbb{P}\otimes\mathbb{P}^{\prime})$. Furthermore we have
$$ 
\widetilde{\mathbb{E}}\left(Z_0(t)-Z_0(s)\right)^2=\mathsf{v}_{2H,1+\ell_{{\delta_{\theta_H}}}^2}(|t-s|):=|t-s|^{2H}\left(1+\ell_{{\delta_{\theta_H}}}^2(|t-s|)\right),
$$
where $\widetilde{\mathbb{E}}$ denotes the expectation under the probability measure $\widetilde{\mathbb{P}}:=\mathbb{P}\otimes \mathbb{P}^{\prime}$. 

Using the assertion \text{i.} of \eqref{cond on L_theta}  and \eqref{slow var lim} we obtain the following
\begin{lemma}
	    There exists a constant $q>1$ such that 
\begin{equation}
	        q^{-1}\,\mathsf{v}_{2H,\, {\ell^2_{\theta_H}}}(h)\leq \widetilde{\mathbb{E}}\left(Z_0(t+h)-Z_0(t)\right)^2 \leq q\,\,\mathsf{v}_{2H,\, {\ell^2_{\theta_H}}}(h),
\end{equation}
for all $h\in [0,1]$ and $t\in [0,1]$. 
\end{lemma}
\noindent For simplicity, we denote by $\mathbf{d}_{H,\ell_{\delta_{\theta_H}}}$ and $\mathbf{d}_{H,(1+\ell_{\delta_{\theta_H}}^2)^{1/2}}$ the canonical metrics of $B^{\delta_{\theta_H}}$ and $Z$ respectively. A consequence of the previous lemma these canonical metrics 
are strongly equivalents to the metric $(s,t)\mapsto \mathbf{d}_{H,{\ell_{\theta_H}}}(t,s):=\mathsf{v}_{H,\,{\ell_{\theta_H}}}(|t-s|)$, leading to the strong equivalence of the metrics $\rho_{\mathbf{d}_{H,\ell_{\delta_{\theta_H}}}}$, $\rho_{\mathbf{d}_{H,(1+\ell_{\delta_{\theta_H}}^2)^{1/2}}}$ and $\rho_{\mathbf{d}_{H,{\ell_{\theta_H}}}}$. 
Hence, their associated capacities 
are also equivalents.

\begin{proof}[Proof of Theorem \ref{theo sharpness}] Let us consider the Gaussian process $Z$ stated above. Using condition \eqref{cond on L_theta} we infer that ${\ell_{\theta_H}(\cdot)}$ satisfies \eqref{concav cond}. Then Lemma \ref{req on slow var} ensures that $\mathsf{v}_{H,{\ell_{\theta_H}(\cdot)}}$ verifies \cite[Hypothesis 2.2]{Hakiki-Viens}. Let $M>0$, applying \cite[Theorem 4.1]{Hakiki-Viens} and the fact that $\rho_{\mathbf{d}_{H,(1+\ell_{\delta_{\theta_H}}^2)^{1/2}}}$ and $\rho_{\mathbf{d}_{H,{\ell_{\theta_H}}}}$ are strongly equivalent,  there exist a positive constant $\mathsf{c}_1$ depending only on $\mathbf{I}$ and $M$ such that 
\begin{equation}\label{lwr bnd for Z}
    \widetilde{\mathbb{P}}\left\{Z(E)\cap F\neq \emptyset\right\}\geq 
    \,\mathsf{c}_1\,\mathcal{C}_{\rho_{\mathbf{d}_{H,{\ell_{\theta_H}}}}}^{d}(E\times F), 
\end{equation}
for any compact sets $E\subset \mathbf{I}$ and $F\subset [-M,M]^d$. Let $ 0< \gamma < d$ and fix a $\gamma$-Ahlfors-David regular set $F_{\gamma}\subset [-M,M]^d$. Then by using \eqref{lwr cap estim 2} with $\rho_1=\mathbf{d}_{H,{\ell_{\theta_H}}}$ and $\rho_3=\rho_{\mathbf{d}_{H,{\ell_{\theta_H}}}}$ and \eqref{uppr Hausdorff estim 2} with $\rho_1=\mathbf{d}_{H}$ and $\rho_3=\rho_{\mathbf{d}_{H}}$, we obtain 
\begin{equation}\label{lwr bnd Z & uppr bnd Haus}
\mathcal{C}_{\rho_{\mathbf{d}_{H,{\ell_{\theta_H}}}}}^{d}(E\times F_{\gamma})\geq\, \mathsf{c}_2^{-1}\, \mathcal{C}_{\mathbf{d}_{H,{\ell_{\theta_H}}}}^{ d-\gamma}(E) \quad \text{ and }\quad \mathcal{H}_{\rho_{\mathbf{d}_{H}}}^d(E\times F_{\gamma})\leq\, \mathsf{c}_{2}\, \mathcal{H}^{H(d-\gamma)}(E),
\end{equation}
for all compact $\,E\subset \mathbf{I}$ and for some constant $\mathsf{c}_2>0$.
Now it is not difficult to see that the functions $h(t):=t^{H(d-\gamma)}$ and  
$$
\Phi(t)= 1/\mathsf{v}^{d-\gamma}_{H,{\ell_{\theta_H}}}(t)=1/\mathsf{v}_{H\left(d-\gamma\right),{\ell_{\theta_H}^{\left(d-\gamma\right)}}}(t),
$$
satisfy the hypotheses of Theorem 4 in \cite{Taylor61} which allow us to conclude that there exists a compact set $E_{\gamma}\subset \mathbf{I}$ such that
\begin{equation}\label{app of taylor theo}
\mathcal{H}^{H(d-\gamma)}(E_{\gamma})=0\quad \quad \text{ and } \quad \quad \mathcal{C}_{\mathbf{d}_{H,{\ell_{\theta_H}}}}^{ d-\gamma}(E_{\gamma})>0.
\end{equation}
Consequently, combining \eqref{lwr bnd Z & uppr bnd Haus} and \eqref{app of taylor theo}, we have 
\begin{equation}\label{hitting Z}
\mathcal{H}_{\rho_{\mathbf{d}_{H}}}^d(E_{\gamma}\times F_{\gamma})=0  \quad \text{ and }\quad  \widetilde{\mathbb{P}}\left\{Z(E_{\gamma})\cap F_{\gamma}\neq \emptyset\right\}\geq\, (\mathsf{c}_1/\mathsf{c}_2)\,  \mathcal{C}_{\mathbf{d}_{H,{\ell_{\theta_H}}}}^{d-\gamma}(E_{\gamma})>0.
\end{equation}

\noindent Now the remainder of the proof is devoted to the construction of a drift $f$ satisfying \eqref{contre exple}. As a consequence of Fubini's theorem, we have 
 \begin{equation*}
\mathbb{E}^{\prime}\left(\mathbb{P}\left\{(B^H+B^{\delta_{\theta_H}}(\omega^{\prime}))(E_{\gamma})\cap F_{\gamma}\neq \emptyset \right\}-\mathsf{c}_3\, \mathcal{C}_{\mathbf{d}_{H,{\ell_{\theta_H}}}}^{d-\gamma}(E_{\gamma})\right)>0,
\end{equation*}
for some fixed positive constant $\mathsf{c}_3\in (0,\mathsf{c}_1/\mathsf{c}_2)$, leading to	
\begin{align}\label{desired paths}
\mathbb{P}^{\prime}\left\{\omega^{\prime}\in \Omega^{\prime}\,:\,\mathbb{P}\left\{(B^H+B^{\delta_{\theta_H}}(\omega^{\prime}))(E_{\gamma})\cap F_{\gamma}\neq \emptyset \right\}> \mathsf{c}_3\,\mathcal{C}_{\mathbf{d}_{H,{\ell_{\theta_H}}}}^{ d-\gamma}(E_{\gamma})\right\}>0.
\end{align}
We therefore choose the function $ f$ among of them. Hence we get the desired result. 
\end{proof}}
    


\section{Hitting points}

As mentioned previously in the introduction our goal in this section is to shed some light on the hitting probabilities for general measurable drift. The resulting estimates are given in the following.    

\noindent 

\begin{theorem}\label{prop hit point}
		Let $\{B^H(t): t\in [0,1]\}$ be a $d$-dimensional fractional Brownian motion with Hurst index $H\in (0,1)$. Let $f: [0,1]\rightarrow \mathbb{R}^d$ be a bounded Borel measurable function and let $E\subset \mathbf{I}$\, be a Borel set. Then for any $M>0$ there exists a constant ${\fontsize{14}{0} \selectfont \textbf{c}}_1\geq 1$ such that for all $x\in [-M,M]^d$ we have
\begin{align}\label{upper-lower hitting points}
{\fontsize{14}{0} \selectfont \textbf{c}}_1^{-1}\mathcal{C}_{\rho_{\mathbf{d}_{H}}}^{d}(Gr_E(f))\leq \mathbb{P}\left\lbrace \exists\, t\in E : (B^H+f)(t)=x\right\rbrace\leq {\fontsize{14}{0} \selectfont \textbf{c}}_1\, \mathcal{H}_{\rho_{\mathbf{d}_{H}}}^{d}(Gr_E(f)).
		\end{align}
\end{theorem}
\noindent The following lemmas {are very valuable} to prove Theorem  \ref{prop hit point} 
\begin{lemma}[Lemma 3.1, \cite{Bierme Lacaux Xiao 2009}]\label{lm estim Gauss}
	Let $\{B^H(t): t\in [0,1]\}$ be a fractional Brownian motion with Hurst index $H\in (0,1)$. {For any constant $M>0$}, there exists positive constants ${\fontsize{14}{0} \selectfont \textbf{c}}_2$ 
	and $\varepsilon_0>0$ such that for all $r\in (0,\varepsilon_0)$, $t\in \mathbf{I}$ and all $x\in [-M,M]^d$,
	\begin{equation}
	\mathbb{P}\left(\inf _{ s\in I,|s-t|^H\leq r}\|B^H(s)-x\| \leqslant r\right) \leqslant {\fontsize{14}{0} \selectfont \textbf{c}}_2\, r^{d}.
	\label{Gaussian estim 1}
	\end{equation}
\end{lemma}
\begin{lemma}[Lemma 3.2,\cite{Bierme Lacaux Xiao 2009} ]\label{l5}
		Let $B^H$ be a fractional Brownian motion with Hurst index $H\in (0,1)$. Then there exists a positive constant ${\fontsize{14}{0} \selectfont \textbf{c}}_3$ such that for all $\epsilon \in (0,1)$, $s,t\in \mathbf{I}$ and $x,y\in \mathbb{R}^d$ we have
\begin{align}\label{Gaussianestim2}
\int_{\mathbb{R}^{2d}}e^{-i(\langle\xi,x\rangle+\langle\eta,y\rangle)}&\exp\left(-\dfrac{1}{2}(\xi,\eta)\left(\varepsilon I_{2d}+\text{Cov}(B^{H}(s),B^{H}(t))\right)(\xi,\eta)^{T}\right)d\xi d\eta \nonumber\\
& \leq \dfrac{{\fontsize{14}{0} \selectfont \textbf{c}}_{3}}{\left(\rho_{\mathbf{d}_{H}}((s,x),(t,y))\right)^{d}},  
\end{align}
where $\Gamma_{\varepsilon}(s,t):=\varepsilon\,I_2+\text{Cov}(B^H_0(s),B^H_0(t))$, $I_{2d}$ and $I_2$ are the identities matrices of order $2d$ and $2$ respectively, $Cov(B^H(s),B^H(t))$ and $Cov(B^H_0(s),B^H_0(t))$ denote the covariance matrix of the random vectors $(B^H(s),B^H(t))$ and $(B^H_0(s),B^H_0(t))$ respectively, and $(\xi,\eta)^T$ is the transpose of the row vector $(\xi,\eta)$.
\end{lemma}
	\begin{proof}[Proof of Theorem \ref{prop hit point}]
		  We start with the upper bound. Choose an arbitrary constant  $\gamma>\mathcal{H}_{\rho_{\mathbf{d}_{H}}}^d(Gr_E(f))$, then there is a covering of $Gr_E(f)$ by balls $\{B_{\rho_{\mathbf{d}_{H}}}((t_i,y_i),r_i), i\geq 1\}$ in $\mathbb{R}_+\times\mathbb{R}^d$ such that 
		\begin{equation}
		Gr_E(f)\subseteq \bigcup_{i=1}^{\infty}B_{\rho_{\mathbf{d}_{H}}}((t_i,y_i),r_i)\quad  \text{ and }\quad  \sum_{i=1}^{\infty}(2r_i)^{d}        \leq \gamma.\label{cover Gr_E(f)}
		\end{equation}
		Let $\delta_0$ and $M$ being the constants given in Lemma \ref{lm estim Gauss}. We assume without loss of generality that $r_i<\delta_0$ for all $i \geq 1$. Let $x\in [-M,M]^d$, it is obvious that
		\begin{align}\label{event covering 2}
		\left\{ \exists s\in E\right.&\left.:(B^{H}+f)(s)=x\right\} \subseteq\\
&\bigcup_{i=1}^{\infty}\left\{ \exists\,\left(s,f(s)\right)\in\left(t_{i}-r_{i}^{1/H},t_{i}+r_{i}^{1/H}\right)\times B(y_{i},r_{i})\text{ s.t. }(B^{H}+f)(s)=x\right\}.\nonumber
		\end{align}
Since for every fixed $i\geq 1$ we have 
\begin{align}\label{event inclu}
\left\{ \exists\,\left(s,f(s)\right)\in\left(t_{i}-r_{i}^{1/H},t_{i}+r_{i}^{1/H}\right)\right.& \left.\times B(y_{i},r_{i})\text{ s.t. }(B^{H}+f)(s)=x\right\} \nonumber\\
	&\subseteq \left\{ \inf_{|s-t_i|^{H}<r_i }\|B^H(s)-x+y_i\|\leq r_i \right\},
		\end{align}
		then we get from \cite[Lemma 3.1]{Bierme Lacaux Xiao 2009} that 
\begin{align}
		\mathbb{P}\left\{ \exists\,\left(s,f(s)\right)\in\left(t_{i}-r_{i}^{1/H},t_{i}+r_{i}^{1/H}\right)\times\right.&\left. B(y_{i},r_{i})\text{ s.t. }(B^{H}+f)(s)=x\right\}\nonumber\\
  & \leq\mathbb{P}\left\{ \inf_{|s-t_{i}|^{H}<r_{i}}\|B^{H}(s)-x+y_{i}\|\leq r_i\right\} \nonumber\\
		& \leq {\fontsize{14}{0} \selectfont \textbf{c}}_{2}\,r_{i}^{d},\label{proba estim 2}
\end{align}
where ${\fontsize{14}{0}\selectfont \textbf{c}}_4$ depends on $H$, $\mathbf{I}$, $M$ and $f$ only. Combining \eqref{cover Gr_E(f)}-\eqref{proba estim 2} we derive that 
$$
\mathbb{P}\left\lbrace \exists s\in E : (B^H+f)(s)=x\right\rbrace\leq 2^{-d}{\fontsize{14}{0} \selectfont \textbf{c}}_{2}\, \gamma.
$$
Let  $\gamma\downarrow \mathcal{H}_{{\rho_{\mathbf{d}_{H}}}}^{d}(Gr_E(f))$, the upper bound in \eqref{upper-lower hitting points} follows.
		
The lower bound in \eqref{upper-lower hitting points} holds from the second moment argument. We assume that $\mathcal{C}_{\rho_{\mathbf{d}_{H}}}^{d}(Gr_E(f))>0$, then let $\sigma$ be a measure supported on $Gr_E(f)$ such that 
		\begin{align}\label{ener<cap}
		\mathcal{E}_{\rho_{\mathbf{d}_{H}},d}(\sigma)=\int_{Gr_E(f)}\int_{Gr_E(f)}\frac{d\sigma(s,f(s))d\sigma(t,f(t))}{\rho_{\mathbf{d}_{H}}((s,f(s)),(t,f(t)))^{d}}\leq \frac{2}{\mathcal{C}_{\rho_{\mathbf{d}_{H}}}^{d}(Gr_E(f))}.
		\end{align}
		Let $\nu$ be the measure on $E$ satisfying $\nu:=\sigma\circ P_1^{-1} $ where $P_1$ is the projection mapping on $E$, i.e. $P_1(s,f(s))=s$. For $n\geq 1$ we consider a family of random measures $\nu_n$ on $E$ defined by
		
		\begin{align} \label{seq rand meas Gr}
		\int_{E} g(s) \nu_{n}(d s) &=\int_{E} (2 \pi n)^{d / 2} \exp \left(-\frac{n\|B^H(s)+f(s)-x\|^2}{2}\right) g(s) \nu(d s) \nonumber
		\\
		&=\int_{E} \int_{\mathbb{R}^{d}} \exp \left(-\frac{\|\xi\|^{2}}{2 n}+i\left\langle\xi, B^H(s)+f(s)-x\right\rangle\right) g(s) d \xi \nu(d s), 
		\end{align}
		where $g$ is an arbitrary measurable function on $\mathbb{R}_+$. Our aim is  to show that  $\left\lbrace \nu_n, n\geq 1\right\rbrace $ has a subsequence which converges weakly to a finite measure $\nu_{\infty}$ supported on the set $\{s\in E : B^H(s)+f(s)=x\}$.
		To carry out this goal, we will start by establishing the following inequalities 
		\begin{align}
		\mathbb{E}(\|\nu_n\|)\geqslant {\fontsize{14}{0} \selectfont \textbf{c}}_{5},\quad \quad \mathbb{E}(\|\nu_n\|^2)\leqslant {\fontsize{14}{0} \selectfont \textbf{c}}_{3} \mathcal{E}_{\rho_{\mathbf{d}_{H}},d}(\sigma),\label{esp energy}
		\end{align}
		which constitute together with the Paley-Zygmund inequality  the cornerstone of the proof.
		Here $\|\nu_n\|$ denotes the total mass of $\nu_n$. By \eqref{seq rand meas Gr}, Fubini's theorem and the use of the characteristic function of a Gaussian vector we have 
		\begin{align}\label{energ-mes1}
		\mathbb{E}(\|\nu_n\|)& =\int_{E}\int_{\mathbb{R}^{d}}e^{-i\left\langle \xi,x-f(s)\right\rangle }\exp\left(-\frac{\|\xi\|^{2}}{2n}\right)\mathbb{E}\left(e^{i\left\langle \xi,B^{H}(s)\right\rangle }\right)\,d\xi\,\nu(ds)\nonumber\\
		& =\int_{E}\int_{\mathbb{R}^{d}}e^{-i\left\langle \xi,x-f(s)\right\rangle }\exp\left(-\frac{1}{2}\left(\frac{1}{n}+s^{2H}\right)\|\xi\|^{2}\right)\,d\xi\,\nu(ds)\nonumber\\
		& =\int_{E}\left(\frac{2\pi}{n^{-1}+s^{2H}}\right)^{d/2}\exp\left(-\frac{\|x-f(s)\|^{2}}{2\left(n^{-1}+s^{2H}\right)}\right)\,\nu(ds)\nonumber\\
		& \geq\int_{E}\left(\frac{2\pi}{1+s^{2H}}\right)^{d/2}\exp\left(-\frac{\|x-f(s)\|^{2}}{2s^{2H}}\right)\,\nu(ds)\nonumber\\
		&\geq {\fontsize{14}{0} \selectfont \textbf{c}}_{5}>0.
		\end{align}  
	Since $f$ is bounded, $x\in [-M,M]^d$ and $\nu$ is a probability measure we conclude that ${\fontsize{14}{0} \selectfont \textbf{c}}_{5}$ is independent of $\nu$ and $n$. This gives the first inequality in \eqref{esp energy}.
	
	We will now turn our attention to the second inequality  in \eqref{esp energy}.  By \eqref{seq rand meas Gr} and Fubini's theorem again we obtain
		\begin{align}
		\mathbb{E}\left(\|\nu_{n}\|^{2}\right)   & =\int_{E}\int_{E}\nu(ds)\nu(dt)\int_{\mathbb{R}^{2d}}e^{-i(\left\langle \xi,x-f(s)\right\rangle +\left\langle \eta,x-f(t)\right\rangle )} \nonumber\\
		& \times\exp(-\frac{1}{2}(\xi,\eta)(n^{-1}I_{2d}+Cov(B^{H}(s),B^{H}(t)))(\xi,\eta)^{T})d\xi d\eta \nonumber\\
		& \leqslant {\fontsize{14}{0} \selectfont \textbf{c}}_{3}\int_{Gr_{E}(f)}\int_{Gr_{E}(f)}\frac{d\sigma(s,f(s))d\sigma(t,f(t))}{(\max\{|t-s|^{H},\|f(t)-f(s)\|\})^{d}}={\fontsize{14}{0} \selectfont \textbf{c}}_{3}\,\mathcal{E}_{\rho_{\mathbf{d}_{H}},d}(\sigma)<\infty,
		\end{align}
where the first inequality  is a direct consequence of  Lemma \ref{l5}. Plugging the moment estimates of \eqref{esp energy} into the Paley–Zygmund inequality (c.f. Kahane \cite{Kahane}, p.8), allows us to confirm that there exists an event  $\Omega_0$ of positive probability such that, for all $\omega\in \Omega_0$, $(\nu_n(\omega))_{n\geq 1}$ admits a subsequence converging weakly to a finite positive measure $\nu_{\infty}(\omega)$  supported on the set $\{s\in E : B^H(\omega,s)+f(s)=x\}$, satisfying the moment estimates in \eqref{esp energy}. Hence we have
\begin{align}
\mathbb{P}\left\{\exists s\in E: (B^H+f)(s)=x\right\}\geq \mathbb{P}\left(\|\nu_{\infty}\|>0\right)\geq \frac{\mathbb{E}(\|\nu_{\infty}\|)^2}{\mathbb{E}(\|\nu_{\infty}\|^2)}\geq \frac{{\fontsize{14}{0} \selectfont \textbf{c}}_5^2}{{\fontsize{14}{0} \selectfont \textbf{c}}_3\mathcal{E}_{\rho_{\mathbf{d}_{H}},d}(\sigma)}.
		\end{align}
		Combining this with \eqref{ener<cap} yields the lower bound in \eqref{upper-lower hitting points}. The proof is completed.
	\end{proof}
 
\begin{remark}\label{hard extension}
We mention that the covering argument used to prove the upper bound in \eqref{upper-lower hitting points} can also serve to show that for any Borel set $F\subset\mathbb{R}^d$, there exists a positive constant ${\fontsize{14}{0} \selectfont \textbf{c}}$ such that 
\begin{align}
\mathbb{P}\left\lbrace (B^H+f)(E)\cap F\neq \emptyset \right\rbrace\leq {\fontsize{14}{0} \selectfont \textbf{c}}\, \mathcal{H}_{\widetilde{\rho}_H}^d(Gr_E(f)\times F).
		\end{align}
		Here $\mathcal{H}_{\widetilde{\rho}_{\mathbf{d}_{H}}}^{\alpha}(\centerdot)$ is the $\alpha$-dimensional Hausdorff measure on the metric space $(\mathbb{R}_+\times\mathbb{R}^d\times\mathbb{R}^d,\widetilde{\rho}_{\mathbf{d}_{H}})$, where the metric ${\widetilde{\rho}_{\mathbf{d}_{H}}}$ is defined by 
\begin{align*}
\widetilde{\rho}_{\mathbf{d}_{H}}((s,x,u),(t,y,v)):=\max\{|t-s|^H,\|x-y\|,\|u-v\|\}.
\end{align*}
		But it seems hard to establish a lower bound in terms of $\mathcal{C}_{\widetilde{\rho}_H}^{d}(Gr_E(f)\times F)$ even for Ahlfors-David regular set $F$.
	\end{remark}
\noindent As a consequence of Theorem \ref{prop hit point}, we obtain a weaker version of \cite[Theorem  3.2]{Erraoui Hakiki} 

\begin{corollary}\label{cor on Leb-image}
 Let $B^H$, $f$, and $E$ as in Theorem \ref{prop hit point}. Then
 \begin{itemize}
     \item[i)] If \,$\,\mathcal{C}_{\rho_{\mathbf{d}_{H}}}^{d}(Gr_E(f))>0$ then \,$\,\lambda_d\left((B^H+f)(E)\right)>0$ with positive probability.
     \item[ii)] If $\mathcal{H}_{\rho_{\mathbf{d}_{H}}}^{d}(Gr_E(f))=0$ then \,$\,\lambda_d\left((B^H+f)(E)\right)=0$ almost surely.
\end{itemize}
\begin{proof}
       Integrating \eqref{upper-lower hitting points} of Theorem \ref{prop hit point} over all cube $[-M,M]^d$, $M>0$ with respect Lebesgue measure $\lambda_d$, we obtain that
     \begin{align}
          (2M)^d\,{\fontsize{14}{0} \selectfont \textbf{c}}_1^{-1}\mathcal{C}_{\rho_{\mathbf{d}_{H}}}^{d}(Gr_E(f))\leq \mathbb{E}\left[\lambda_d\left([-M,M]^d\cap (B^H+f)(E)\right)\right]\leq  (2M)^d\,{\fontsize{14}{0} \selectfont \textbf{c}}_1\, \mathcal{H}_{\rho_{\mathbf{d}_{H}}}^{d}(Gr_E(f)).
     \end{align}
Therefore if $\mathcal{C}_{\rho_{\mathbf{d}_{H}}}^{d}(Gr_E(f))>0$ we obtain  
$$
\mathbb{E}\left[\lambda_d\left((B^H+f)(E)\right)\right]>0.
$$ 
Hence $\lambda_d\left((B^H+f)(E)\right)>0$ with positive probability, which finishes the proof of (i). On the other hand, if $\mathcal{H}_{\rho_{\mathbf{d}_{H}}}^{d}(Gr_E(f))=0$ we obtain that $\lambda_d\left([-n,n]^d \cap (B^H+f)(E)\right)=0$ a.s. for all $n\in \mathbb{N}^{*}$. Then we have $\lambda_d\left((B^H+f)(E)\right)=0$ a.s. Hence the proof of (ii) is completed.
\end{proof}
\end{corollary}

\begin{remark}
\textbf{(i)} Let $\dim_{\rho_{\mathbf{d}_{H}}}(\cdot)$ be the Hausdorff dimension in the metric space  $( \mathbb{R}_+\times \mathbb{R}^d,\rho_{\mathbf{d}_{H}})$ defined in \eqref{def Haus dim}. There is a close relationship between $\dim_{\rho_{\mathbf{d}_{H}}}(\cdot)$ and $H$-parabolic Hausdorff dimension $\dim_{\Psi, H}(\cdot)$, used in Peres and Sousi \cite{Peres&Sousi2016} and in Erraoui and Hakiki \cite{Erraoui Hakiki}, expressed by
$$
\dim_{\Psi, H}(.)\equiv H \times\dim_{\rho_{\mathbf{d}_{H}}}(.).
$$
See \cite[Remark 2.2]{Erraoui Hakiki}.
\\
\textbf{(ii)} The previous corollary is a weaker version of \cite[Theorem  3.2]{Erraoui Hakiki} in the following sense: According to Theorem 1.2. in \cite{Peres&Sousi2016} and \textbf{(i)} we have 
$$
\dim(B^H+f)(E)=\frac{\dim_{\Psi,H}(Gr_E(f))}{H}\wedge d=\dim_{\rho_{\mathbf{d}_{H}}}(Gr_E(f))\wedge d.
$$
Therefore \cite[Theorem 3.2.]{Erraoui Hakiki} asserts that, if $\dim_{\rho_{\mathbf{d}_{H}}}(Gr_E(f))>d$ then $\lambda_d(B^H+f)(E))>0$ almost surely. On the other hand, Corollary \ref{cor on Leb-image} (i) ensures, under the condition $\mathcal{C}_{\rho_{\mathbf{d}_{H}}}^{d}(Gr_E(f))>0$, that $\lambda_d(B^H+f)(E))>0$ only with positive probability. It is well known from Frostman's Lemma that the {condition} $\mathcal{C}_{\rho_{\mathbf{d}_{H}}}^{d}(Gr_E(f))>0$ is weaker than $\dim_{\rho_{\mathbf{d}_{H}}}(Gr_E(f))>d$.
\end{remark}

\section{Application to polarity}

{Let} $E\subset \mathbf{I}$. We say that a point $x\in \mathbb{R}^d$ is polar for $(B^H+f)\vert_{E}$, the restriction of $(B^H+f)$ to $E$, if 
\begin{equation}\label{def of polarity}
	\mathbb{P}\left\{ \exists \, t\in E\,:\, (B^H+f)(t)=x\right\}=0.
\end{equation}
Otherwise, $x$ is said to be non-polar for $(B^H+f)\vert_{E}$. In other words, $(B^H+f)\vert_{E}$ hits the point $x$. 

\noindent It is noteworthy that, when $f\in C^{H}(\mathbf{I})$, the hitting probabilities estimates in \eqref{upper-lower hitting points} becomes 
\begin{equation}\label{hits pts of B^H regular drift f}
	{\fontsize{14}{0} \selectfont \textbf{c}}_1^{-1}\mathcal{C}^{Hd}(E)\leq \mathbb{P}\left\lbrace \exists t\in E : (B^H+f)(t)=x\right\rbrace\leq {\fontsize{14}{0} \selectfont \textbf{c}}_1\, \mathcal{H}^{Hd}(E).
\end{equation}
See also Corollary 2.8 in \cite{Chen2013}. Consequently, all points are non-polar (resp. polar) for $(B^H+f)\vert_{E}$ when $\dim(E)>Hd$ (resp. $\dim(E)<Hd$). However, the critical dimensional case, which is the most important and not easy to deal with, is $\dim(E)=Hd$. This is undecidable in general as illustrated in the following
\begin{proposition}\label{indecid-crital-cas}
	Let $\{B^{H}(t) : t \in [0,1]\}$ be a $d$-dimensional fractional Brownian motion of Hurst index $H\in (0,1)$ such that $H d<1$. Let $f:[0,1]\rightarrow \mathbb{R}^d$ be a $H$-Hölder continuous function and $E\subset \mathbf{I}$ be a Borel set. Then
	there exist two Borel subsets $E_1$ and $E_2$  of $\mathbf{I}$ such that $\dim(E_1)=\dim(E_2)=H\,d$ and for all $x\in \mathbb{R}^d$ we have  
	$$
	\mathbb{P}\left\{\exists s\in E_1\,:\, (B^H+f)(s)=x\right\}=0 \quad \text{and} \quad \mathbb{P}\left\{\exists s\in E_2\,:\, (B^H+f)(s)=x\right\}>0.        
	$$
\end{proposition}

\noindent The following Lemma is {helpful} in the proof of Proposition \ref{indecid-crital-cas}.

\begin{lemma}\label{main prop critic hits pts}
Let $\alpha\in (0,1)$ and $\beta>1$. {Let $E_1$ and $E_2$ are two Borel subsets of $\,\mathbf{I}$ supporting two probability measures $\nu_1$ and $\nu_2$ respectively, that satisfying}
\begin{equation}\label{modif Dav-Ahlf reg 1}
	\mathsf{c}_1^{-1} r^{\alpha}\log^{\beta}(e/r)\leq \nu_1\left([a-r,a+r]\right)\leq \mathsf{c}_1 r^{\alpha}\log^{\beta}(e/r)\quad \text{ for all $r\in (0,1)$, $a\in E_1$},
\end{equation}
and 
\begin{equation}\label{modif Dav-Ahlf reg 2}
	\mathsf{c}_2^{-1} r^{\alpha}\log^{-\beta}(e/r)\leq \nu_2\left([a-r,a+r]\right)\leq \mathsf{c}_2 r^{\alpha}\log^{-\beta}(e/r)\quad \text{ for all $r\in (0,1)$, $a\in E_2$},
\end{equation}
for some positive constants $\mathsf{c}_1$ and $\mathsf{c}_2$. Then we have $$\dim(E_1)=\dim(E_2)=\alpha$$ and 
\begin{equation}\label{singular sets 1}
	\mathcal{H}^{\alpha}(E_1)=0 \quad\text{and}\quad \mathcal{C}^{\alpha}(E_2)>0. 
\end{equation}
See Appendix B\, for examples of such measures $\nu_1$ and $\nu_2$.
\begin{proof}
	First, let us start by proving that $\dim(E_1)=\dim(E_2)=\alpha$. Indeed, for all $t \in  [0,1]$ and $n\in \mathbb{N}$ we denote by $I_n(t)$ the $n$th generation, half open dyadic interval of the form $[\frac{j-1}{2^n},\frac{j}{2^n})$ containing $t$. Then, by using \eqref{modif Dav-Ahlf reg 1} and \eqref{modif Dav-Ahlf reg 2}, it is easy to check that 
	\begin{equation}
		\lim_{n\rightarrow \infty}\frac{\log \nu_1(I_n(t))}{\log(2^{-n})}=\lim_{n\rightarrow \infty}\frac{\log \nu_2(I_n(t))}{\log(2^{-n})}=\alpha.    
	\end{equation} 
	Therefore, by Billingsley Lemma \cite[Lemma 1.4.1 p. 16]{Bishop Peres} we have $\dim(E_1)=\dim(E_2)=\alpha$. 
	
	\noindent {Now we are going to look at} \eqref{singular sets 1}. Let $r\in (0,1]$ and $P_{|.|}(E_1,r)$ be the packing number of $E_1$.
 The lower bound in \eqref{modif Dav-Ahlf reg 1} leads via
	$$ 
	\mathsf{c}_1^{-1}\,r^{\alpha}\, \log^{\beta}(e/r)\, P_{|.|}(E_1,r)\leq \sum_{j=1}^{P_{|.|}(E_1,r)}\nu_1(I_j)=\nu_1(E_1)=1,
	$$
	{to 
		\[
		P_{|.|}(E_{1},r)\leq\mathsf{c}_{1}^{\,}r^{-\alpha}\,\log^{-\beta}(e/r).
		\]
	}
	So using the well known fact that $N_{|.|}(E_1,2r)\leq P_{|.|}(E_1,r)$, we may deduce that 
	$$N_{|.|}(E_1,r)\leq \mathsf{c}_1\, 2^{\alpha}\,r^{-\alpha}\, \log^{-\beta}(2e/r),\,\, \forall r\in (0,1).$$ 
	Therefore, it is not hard to make out that 
	$$
	\mathcal{H}^{\alpha}(E_1)\leq \, \limsup_{r\rightarrow 0}(2r)^{\alpha}\, N_{|.|}(E_1,r)=0,
	$$
	which gives the first outcome of \eqref{singular sets 1}.
	On the other hand,  to show that $\mathcal{C}^{\alpha}(E_2)>0$ it is sufficient to prove that $$\sup_{t\in E_2}\dint_{E_2}\frac{\nu_2(ds)}{|t-s|^{\alpha}}<\infty.$$ Indeed, we first assume without loss of generality that $\kappa=\operatorname{diam}(E_2)<1$. Now since $\nu_2$  has no atom, then for {any} $t  \in  E_2$ we have 
	\begin{align}\label{estim energy critical set hits pts}
		\dint_{E_2}\frac{\nu_2(ds)}{|t-s|^{\alpha}}=\sum_{j=0}^{\infty}  \dint_{\{s:|t-s|\in (\kappa 2^{-(j+1)},\kappa 2^{-j}]\}}\frac{\nu_2(ds)}{|t-s|^{\alpha}} &\leq \sum_{j=0}^{\infty} \kappa^{-\alpha} 2^{\alpha\,(j+1)}\nu_2\left([t-\kappa\,2^{-j},t+\kappa\,2^{-j}]\right)\nonumber\\
		&\leq 2^{\alpha}\mathsf{c}_2\,\sum_{j=0}^{\infty}\frac{1}{\log^{\beta}(e\,2^j/\kappa)\,}, 
	\end{align}
	which is finite independently of $t$. Hence $\mathcal{E}_{\alpha}(\nu_2)<\infty$ and therefore $\mathcal{C}^{\alpha}(E_2)>0$.
\end{proof}
\end{lemma}
\begin{proof}[Proof of Proposition \ref{indecid-crital-cas}] 
It is a direct consequence of \eqref{hits pts of B^H regular drift f} and Lemma \ref{main prop critic hits pts} with $\alpha=H d$. 
\end{proof}

Against this background, it is worthy to note that, according to \eqref{hits pts of B^H regular drift f} and Proposition \ref{indecid-crital-cas}, the $H$-Hölder regularity of the drift $f$ is insufficient to guarantee the non-polarity of points for $(B^H+f)\lvert_E$ for a Borel set $E\subset [0,1]$ such that $\dim(E)= Hd$ which implicitly involves the need for a bite of drift roughness. Namely, the drift $f$ will be chosen, as in the previous section, to be ($H-\varepsilon$)-Holder continuous for all $\varepsilon>0$ without reaching order $H$. On the other hand in accordance to Theorem \ref{prop hit point}, for a general measurable drift $f$, a sufficient condition for $(B^H+f)\vert_{E}$ to hits all points is $\mathcal{C}_{\rho_{\mathbf{d}_{H}}}^{d}(Gr_E(f))>0$. The overall point of what follows is to provide some examples of drifts satisfying this last condition with a Borel set $E$ whose distinctive feature is $\dim(E)=Hd$.

Given a slowly varying function at zero $\ell:(0,1]\,\rightarrow \R_+$ we associate to it the  kernel $\Phi_{H,\ell}(\cdot)$ defined as follows
\begin{equation}\label{kernel}
	\Phi_{H,\ell}(r):=r^{-Hd}\,\ell^{-d}(r)\,\left(1+\log\left(1\vee\ell(r)\right)\right)
\end{equation}
%
{Let $\pmb{\theta_H}$ be the normalized regularly varying function defined in \eqref{inv spectr density rep} with a normalized slowly varying part  $L_{\pmb{\theta_H}}$. As previously, we consider the normalized regularly varying function at zero $\delta_{\pmb{\theta_H}}$, with index $H$ and normalized slowly varying part  $\ell_{\delta_{\pmb{\theta_H}}}$, satisfying
\eqref{equiv increm var} and, on the space $(\Omega^{\prime},\mathcal{F}^{\prime},\mathbb{P}^{\prime})$,  the $d$-dimensional Gaussian process with stationary increments $B^{\delta_{\pmb{\theta_H}}}(t):=\left( B^{\delta_{\pmb{\theta_H}}}_1(t),....,B^{\delta_{\pmb{\theta_H}}}_d(t)\right), t\in [0,1]$, where $B^{\delta_{\pmb{\theta_H}}}_1,...,B^{\delta_{\pmb{\theta_H}}}_d$ are $d$ independent copies of $B^{\delta_{\pmb{\theta_H}}}_0$ defined in \eqref{spectral representation}. The following lemma will be useful afterwards

\begin{lemma}\label{estim kernel 1}
	
	There exists a positive constant ${\mathsf{c}}_3$ such that 
	\begin{equation}\label{cap-estimate 1}
		\mathbb{E}'\left[\left(\max\left\lbrace t^H,\lVert B^{\delta_{\pmb{\theta_H}}}(t)\rVert \right\rbrace \right)^{-d}\right]\leq \mathsf{c}_3\, \Phi_{H,\ell_{\delta_{\pmb{\theta_H}}}}(t), \quad \forall t\in (0,t_0),
	\end{equation}
	for some $t_0\in (0,1)$.
\end{lemma}
%
\begin{proof}First we note that, since $B^{\delta_{\pmb{\theta_H}}}(t)$ is a $d$-dimensional Gaussian vector, the term on the left hand side of \eqref{cap-estimate 1} has the same distribution as $t^{-Hd}\left(\max\left\lbrace 1,\ell_{\delta_{\pmb{\theta_H}}}(t)\lVert N\rVert\right\rbrace \right)^{-d}$, where $N$ is a $d$-dimensional standard normal random vector. Due to simple calculations we obtain
	\[
	\begin{array}{l}
		\mathbb{E}'\left[\left(\max\left\lbrace 1,\ell_{\delta_{\pmb{\theta_{H}}}}(t)\lVert N\rVert\right\rbrace \right)^{-d}\right]
		=\mathbb{P}'\left[\lVert N\rVert\leq\ell_{\delta_{\pmb{\theta_{H}}}}^{-1}(t)\right]+\ell_{\delta_{\pmb{\theta_{H}}}}^{-d}(t)\mathbb{E}'\left[\dfrac{1}{\lVert N\rVert^{d}}1\!\!1_{\{\lVert N\rVert>\ell_{\delta_{\pmb{\theta_{H}}}}^{-1}(t)}\right]\\
		= \left(2\pi\right)^{-d/2}\,\left(\dint_{\left\{ \lVert y\rVert\leq\ell_{\delta_{\pmb{\theta_{H}}}}^{-1}(t)\right\} }\,e^{-\lVert y\rVert^{2}/2}dy+\ell_{\delta_{\pmb{\theta_{H}}}}^{-d}(t)\dint_{\left\{ \lVert y\rVert>\ell_{\delta_{\pmb{\theta_{H}}}}^{-1}(t)\right\} }\,\frac{e^{-\lVert y\rVert^{2}/2}}{\lVert y\rVert^{d}}dy\right)\\
		\\
		\leq\mathsf{c}_{4}\,\ell_{\delta_{\pmb{\theta_{H}}}}^{-d}(t)\left(1+\dint_{\ell_{\delta_{\pmb{\theta_{H}}}}^{-1}(t)}^{\infty}\dfrac{e^{-r^{2}/2}}{r}\,dr\right)
		\leq\mathsf{c}_{4}\,\ell_{\delta_{\pmb{\theta_{H}}}}^{-d}(t)\left(1+\dint_{1\wedge\ell_{\delta_{\pmb{\theta_{H}}}}^{-1}(t)}^{1}r^{-1}\,dr+\dint_{1}^{\infty}\frac{e^{-r^{2}/2}}{r}\,dr\right)\\
		\\
		\leq\mathsf{c}_{4}\,\ell_{\delta_{\pmb{\theta_{H}}}}^{-d}(t)\left(1-\log\left(1\wedge\ell_{\delta_{\pmb{\theta_{H}}}}^{-1}(t)\right)\right)\leq\mathsf{c}_{4}\,\ell_{\delta_{\pmb{\theta_{H}}}}^{-d}(t)\left(1+\log\left(1\vee\ell_{\delta_{\pmb{\theta_{H}}}}(t)\right)\right).
	\end{array}
	\]
\end{proof}

\begin{lemma}\label{positive capacity}
	Let $E$ be a Borel set of $[0,1]$. If $\mathcal{C}_{\Phi_{H,\ell_{\delta_{\pmb{\theta_H}}}}}(E)>0$, then $\mathbb{P}'$-almost surely $$\mathcal{C}_{\rho_{\mathbf{d}_{H}}}^{d}\left(Gr_E(B^{\delta_{\pmb{\theta_H}}})\right)>0.$$ 
\end{lemma}
\begin{proof}
	Firstly, the assumption $\mathcal{C}_{\Phi_{H,\ell_{\delta_{\pmb{\theta_H}}}}}(E)>0$ ensures that there exists a probability measure $\nu$ supported on $E$ with finite energy, i.e.
	$$
	\mathcal{E}_{\Phi_{H,\ell_{\delta_{\pmb{\theta_H}}}}}(\nu)=\dint_E\dint_E \Phi_{H,\ell_{\delta_{\pmb{\theta_H}}}}(|t-s|)\nu(dt)\nu(ds)<\infty.
	$$
	Let $\mu_{\omega^{\prime}}$ be the random measure defined on $Gr_E(B^{\delta_{\pmb{\theta_H}}})$ by 
	$$ 
	\mu_{\omega'}(G):=\nu\{s: (s,B^{\delta_{\theta_0}}(\omega',s))\in G\} \quad \text{ for all } G\subset Gr_E(B^{\delta_{\pmb{\theta_H}}}(\cdot,\omega')).
	$$
	Hence $\mathbb{P}'$-almost surely we have
	\[
	\begin{array}{c}
		\mathcal{E}_{\rho_{\mathbf{d}_{H}},d}(\mu_{\omega'}):=\\ 
		\\	\dint_{Gr_{E}(B^{\delta_{\pmb{\theta_{H}}}})}\dint_{Gr_{E}(B^{\delta_{\pmb{\theta_{H}}}})}\frac{1}{\max\left\{ \lvert t-s\lvert^{Hd},\lVert B^{\delta_{\pmb{\theta_{H}}}}(t)-B^{\delta_{\pmb{\theta_{H}}}}(s)\rVert^{d}\right\} }d\mu_{\omega'}(s,B^{\delta_{\pmb{\theta_{H}}}}(s))d\mu_{\omega'}(t,B^{\delta_{\pmb{\theta_{H}}}}(t))\\
		\\
		=\dint_{E}\dint_{E}\frac{1}{\max\left\{ \lvert t-s\lvert^{Hd},\lVert B^{\delta_{\pmb{\theta_{H}}}}(t)-B^{\delta_{\pmb{\theta_{H}}}}(s)\rVert^{d}\right\} }\nu(ds)\nu(dt).
	\end{array}
	\]
	Therefore, in order to achieve the goal it is sufficient to show that $\mathcal{E}_{\rho_{\mathbf{d}_{H}},d}(\mu_{\omega'})<\infty$ for $\mathbb{P}'$-almost surely, which can be done by checking  $\mathbb{E}'\left[\mathcal{E}_{\rho_{\mathbf{d}_{H}},d}(\mu_{\omega'})\right]<\infty$. Indeed, using Fubini's theorem with the stationarity of increments and Lemma \ref{estim kernel 1} we obtain 
	\begin{align}\label{trans-fub-stat increm}
		\begin{aligned}
			\mathbb{E}'\left[ \mathcal{E}_{\rho_{\mathbf{d}_{H}},d}(\mu_{\omega'})\right] 
			&= \mathbb{E}'\left[\dint_{E}\dint_{E}\frac{1}{\max\left\{ \lvert t-s\lvert^{Hd},\lVert B^{\delta_{\pmb{\theta_{H}}}}(t)-B^{\delta_{\pmb{\theta_{H}}}}(s)\rVert^{d}\right\} }\nu(ds)\nu(dt)\right]\\
			&=\dint_E\dint_E \mathbb{E}'\left[\frac{1}{\max\left\{\lvert t-s\lvert^{Hd}, \lVert B^{\delta_{\pmb{\theta_{H}}}}(|t-s|)\rVert^d \right\}}\right]\nu(ds)\nu(dt)\\
			&\leq \mathsf{c}_3\,\mathcal{E}_{\Phi_{H,\ell_{\delta_{\pmb{\theta_H}}}}}(\nu)<\infty.
		\end{aligned}
	\end{align}
	Thus $\mathcal{C}_{\rho_{\mathbf{d}_{H}}}^{d}\left(Gr_E(B^{\delta_{\pmb{\theta_H}}})\right)>0$  $\mathbb{P}'$-almost surely. 
\end{proof}
\begin{remark}
Notice that in both of Lemmas \ref{estim kernel 1} and \ref{positive capacity} we lose nothing by changing $\ell_{\delta_{\theta_H}}(\cdot)$ by $\ell_{\theta_H}(\cdot)=\mathsf{c}_{H}\,L_{\theta_{H}}^{-1/2}(1/\cdot)$, due to the fact that $\ell_{\delta_{\theta_H}}(h)\thicksim \ell_{\theta_H}(h)$ as $h\rightarrow 0$.  
\end{remark}
The following result is the consequence of the two previous lemmas.

\begin{proposition}\label{prop cont drift 1}
	Let $\{B^{H}(t) : t \in [0,1]\}$ be a $d$-dimensional fractional Brownian motion of Hurst index $H\in (0,1)$. If $\mathcal{C}_{\Phi_{H,\ell_{\delta_{\pmb{\theta_H}}}}}(E)>0$, then there exists a continuous function $f\in C^{\mathsf{w}_{H,{\ell_{\theta_H}}}}(\mathbf{I})\setminus C^{H}\left(\mathbf{I}\right)$ such that
	\begin{align}
		\mathbb{P}\left\{\exists t\in E : (B^{H}+f)(t)=x\right\}>0,
	\end{align}
	for all $x\in \mathbb{R}^d$.
\end{proposition}

\begin{proof} We start the proof by recalling that Theorem \ref{modulus of continuity} provides the modulus of continuity of $B^{\delta_{\pmb{\theta_H}}}$ that is $B^{\delta_{\pmb{\theta_H}}}\in C^{\mathsf{w}_{H, {\ell_{\theta_H}}}}(\mathbf{I})$, $\mathbb{P}'$-almost surely. Now applying Theorem \ref{prop hit point} we deduce that for $\mathbb{P}'$-almost surely there is a positive random constant $C=C(\omega')>0$ such that 
	$$
	\mathbb{P}\left\{ \exists s\in : (B^H+B^{\delta_{\pmb{\theta_H}}}(\omega'))(s)=x \right\}\geq C\, \mathcal{C}_{\rho_{\mathbf{d}_{H}}}^{d}\left(Gr_E\,B^{\delta_{\pmb{\theta_H}}}(\cdot,\omega')\right)>0.
	$$
	Hence, by choosing $f$ to be one of the trajectories of $B^{\delta_{\pmb{\theta_H}}}$ we get the desired result.
\end{proof}
\begin{remark}
It is worthwhile mentioning that $C^{\mathsf{w}_{H,{\ell_{\theta_H}}}}(\mathbf{I})$ verifies \eqref{compar spaces} as \,${{\ell_{\theta_H} (\cdot)}}$ meets the condition  \eqref{cond modulus cont} via the assertion \textbf{i.} of \eqref{cond on L_theta}.
\end{remark}

\noindent Consequently we have the following outcome

\begin{corollary}\label{prop cont drift 2}
	Let $\{B^{H}(t) : t \in [0,1]\}$ be a $d$-dimensional fractional Brownian motion of Hurst index $H\in (0,1)$. Assume that $L_{\pmb{\theta_H}}$, the normalized slowly varying part  of $\pmb{\theta_H}$, satisfies 
\eqref{cond on L_theta}. Then there exist a Borel set $E\subset \mathbf{I}$, such that 
$$
\dim(E)=Hd \quad \text{ and }\quad \mathcal{H}^{Hd}(E)=0,
$$
and a function $f\in C^{\mathsf{w}_{H,{\ell_{\theta_H}}}}(\mathbf{I})\setminus C^{H}\left(\mathbf{I}\right)$, such that all points are non-polar for $(B^H+f)\vert_E$.
\end{corollary}
\begin{remark}
	Similarly to Theorem \ref{theo sharpness}, Corollary \ref{prop cont drift 2} confirms also the sharpness of the H\"older regularity assumption made on the drift $f$ in \eqref{hits pts of B^H regular drift f}. 
\end{remark}
\begin{proof}[Proof of Corollary \ref{prop cont drift 2} ]
	The condition \ref{ppty of slow var part} with $\underset{x\rightarrow +\infty}{\liminf} \,L_{\pmb{\theta_H}}(x)=0$ 
 imply that 
 $$
 \underset{x\rightarrow 0}{\limsup} \,\ell_{\delta_{\pmb{\theta_H}}}(x)= †\underset{x\rightarrow 0}{\limsup} \,\ell_{\pmb{\theta_H}}(x)=+\infty.
 $$ 
 Thus applying once again Theorem 4 in \cite{Taylor61} with the functions $h(t):=t^{H d}$ and  $\Phi(t)=\Phi_{H,\ell_{\delta_{\pmb{\theta_H}}}}(t)$ we infer that there exists a compact set $E\subset \mathbf{I}$ such that  
	$$
	\mathcal{H}^{H d}(E)=0 \quad \text{ and } \quad \mathcal{C}_{\Phi_{H,\ell_{\delta_{\pmb{\theta_H}}}}}(E)>0.
	$$
	Finally Proposition \ref{prop cont drift 1} gives us the function that we are looking for, that is $f\in C^{\mathsf{w}_{H,{\ell_{\theta_H}}}}(\mathbf{I})$  for which $(B^H+f)\vert_E$ hits all points.
\end{proof}
}

\section{Appendixes}


\subsection{Appendix A}
In this section, we would like to investigate existence of a Gaussian process with stationary increments $B^{\delta_{\theta}}$ with increments variance $\delta_{\theta}(\cdot)$, we also provides the uniform modulus of continuity for $B^{\delta_{\theta}}$. 
First, let $\alpha \in (0,1)$ and $\theta_{\alpha}:\mathbb{R}_+\rightarrow \mathbb{R}_+$ be a $\mathcal{C}^{\infty}$ normalized regularly varying function at infinity with index $2\alpha+1$ in the sens of Karamata of the form
\begin{align}\label{inv spectr density rep}
\theta_{\alpha}(x)=x^{2\alpha+1}\, L_{\theta_{\alpha}}(x),
\end{align}
with $L_{\theta_{\alpha}}(\cdot)$ is the normalized slowly varying part given as follows 
\begin{align}\label{slow var part of teta}
L_{\theta_{\alpha}}(x)=\mathtt{c}_1\exp\left(\dint_{x_0}^x\frac{\varepsilon_{\theta_{\alpha}}(t)}{t}dt\right) \quad \text{for all $x\geq x_0$},
\end{align}
and $L_{\theta_{\alpha}}(x)=\mathtt{c}_1$ for all $x\in (0,x_0)$, where $\mathtt{c}_1$ is a positive constant. It is quit simple to check that 
$$ 
\varepsilon_{\theta_{\alpha}}(x)=x\,L_{\theta_{\alpha}}^{\prime}(x)/L_{\theta_{\alpha}}(x) \quad \text{for all $x\geq x_0$}. 
$$
Let $\delta_{\theta_{\alpha}}: \mathbb{R}_+\rightarrow \mathbb{R}_+$ be the continuous function defined by 
\begin{align}\label{increments variance}
\delta^2_{\theta_{\alpha}}(h):=\frac{2}{\pi}\dint_{0}^{\infty}(1-\cos(x\,h))\frac{dx}{\theta_{\alpha}(x)}=\frac{4}{\pi}\dint_{0}^{\infty}\sin^2\left(\frac{x\,h}{2}\right)\frac{dx}{\theta_{\alpha}(x)}.
\end{align}
The special properties of the function $\theta_{\alpha}$ make it easy to verify that $\delta_{\theta_{\alpha}}$ is well-defined. Moreover, it follows from \cite[Theorem 7.3.1]{Marcus-Rosen} that $\delta^2_{\theta_{\alpha}}$ is a normalised regularly varying function at zero with index $2\alpha$ such that
\begin{align}\label{equiv increm var}
{\delta^2_{\theta_{\alpha}}(h)}\thicksim {h^{2\alpha}\,\ell^2_{\theta_{\alpha}}(h)} \quad \text{ as $h\rightarrow 0$},
\end{align}
where 
\begin{align}\label{cst of equiv slow var fcts}
{\ell_{\theta_{\alpha}}(h):=\mathtt{c}^{1/2}_{\alpha}\,L^{-1/2}_{\theta_{\alpha}}(1/h)}\quad \text{ and }\quad \mathtt{c}_{\alpha}:=\frac{4}{\pi}\dint_0^{\infty}\frac{\sin^2 s/2}{s^{2\alpha+1}}ds
\end{align}
Hence $\delta_{\theta_{\alpha}}$ is normalised regularly varying at zero of index $\alpha$, whence there exists $\ell_{\delta_{\theta_{\alpha}}}:(0,1]\rightarrow \mathbb{R}_+$ be normalised slowly varying at zero such that $\delta_{\theta_{\alpha}}(h)=h^{\alpha}\,\ell_{\delta_{\theta_{\alpha}}}(h)=\mathsf{v}_{\alpha,\ell_{\delta_{\theta_{\alpha}}}}(h)$ for all $h\in [0,1]$. Therefore \eqref{equiv increm var} ensures that 
\begin{align}\label{ppty of slow var part}
\quad \ell_{\delta_{\theta_{\alpha}}}(h) \thicksim \ell_{\theta_{\alpha}}(h)=\mathsf{c}_{\alpha}^{1/2} L_{\theta_{\alpha}}^{-1/2}(1/h)\quad \text{as $h \rightarrow 0$}.
\end{align}

In the following, we give a method for constructing real Gaussian centered processes with stationary increments such that their increments variance are given by $\delta_{\theta_{\alpha}}^2$. 
\begin{proposition}\label{Construc of stat increm process}
Let $\alpha \in (0,1)$ and $\theta_{\alpha}$ be the normalised regularly varying function given in \eqref{inv spectr density rep}. Then there exists a one-dimensional centered Gaussian process $B_0^{\delta_{\theta_{\alpha}}}$ on $\mathbb{R}_+$ such that for all 
\begin{align}
	\delta_{\theta_{\alpha}}^2(h)=\mathbb{E}\left(B_0^{\delta_{\theta_{\alpha}}}(t+h)-B_0^{\delta_{\theta_{\alpha}}}(t)\right)^2 \quad \text{ for all $t\geq 0$ and $h\geq 0$}.
\end{align}
\end{proposition}
\begin{proof}
First, let $\mu$ be the measure on $\R_+$ defined by $\mu(dx)=\left(\pi \theta_{\alpha}(x)\right)^{-1}\textbf{1}_{\R_+}(x)\,dx$. Let $\mathcal{W}_1$ and $\mathcal{W}_2$ be two independent Brownian motion on $\R_+ $. 
Now we consider the Gaussian processes $B_0^{\delta_{\theta_{\alpha}}}$ represented as follows
\begin{align}\label{spectral representation}
	B^{\delta_{\theta_{\alpha}}}_0(t):=\dint_0^{\infty}\frac{(1-\cos xt )}{\left(\pi\,\theta_{\alpha}(x)\right)^{1/2}}\,\mathcal{W}_1(dx)+\dint_0^{\infty} \frac{\sin xt}{\left(\pi\,\theta_{\alpha}(x)\right)^{1/2}}\,\mathcal{W}_2(dx).
\end{align}
A simple calculation gives
\begin{align}\label{conseq spec rep 1}
	\mathbb{E}\left[(B^{\delta_{\theta_{\alpha}}}_0(t+h)-B^{\delta_{\theta_{\alpha}}}_0(t))\right]^2=2\dint_{0}^{\infty}(1-\cos(x\,h))\mu(dx)=\delta^2_{\theta_{\alpha}}(h).
\end{align}
\end{proof}
Using \eqref{equiv increm var} we get the following useful estimates helping us to provide uniform modulus of continuity of $B^{\delta_{\theta_{\alpha}}}$ : there exist $h_0>0$ and a constant $q\geq 1$  such that  
\begin{align}\label{conseq bnds increm var}
q^{-1}\,h^{2\alpha}\,{\ell_{\theta_{\alpha}}^2(h)} \leq \mathbb{E}\left[B^{\delta_{\theta_{\alpha}}}_0(t+h)-B^{\delta_{\theta_{\alpha}}}_0(t)\right]^2\leq q\,h^{2\alpha}\,{\ell_{\theta_{\alpha}}^2(h)}.
\end{align}
for all $t\in \R_+$ and $h\in [0,h_0)$.

Notice that, due to \eqref{conseq bnds increm var}, all results of our interest are not sensitive to changing $\ell_{\delta_{\theta_{\alpha}}}$ by $\ell_{\theta_{\alpha}}$. On the other hand, using $\ell_{\theta_{\alpha}}$ instead of $\ell_{\delta_{\theta_{\alpha}}}$ is especially important when the regularity condition \eqref{cond on L_theta} on $L_{\theta}(\cdot)$ is needed. Such condition leads to \eqref{concav cond} for $\ell_{\theta}(\cdot)$, which is required for Lemma \ref{req on slow var}. The following result is about the uniform modulu of continuity of the Gaussian process $B^{\delta_{\theta_{\alpha}}}$
\begin{theorem}\label{modulus of continuity}
Let $B^{\delta_{\theta_{\alpha}}}:=(B^{\delta_{\theta_{\alpha}}}_1(t),....,B^{\delta_{\theta_{\alpha}}}_d(t))$ be a $d$-dimensional Gaussian process, where $B^{\delta_{\theta_{\alpha}}}_1,...,B^{\delta_{\theta_{\alpha}}}_d$ are $d$ independent copies of $B^{\delta_{\theta_{\alpha}}}_0$. Let $\,0<a<b< 1$ and  $\mathbf{I}:=[a,b]$. Then $B^{\delta_{\theta_{\alpha}}}\in C^{\mathsf{w}_{{\alpha,{\ell_{\theta_{\alpha}}}}}}(\mathbf{I})$ a.s.  with $\mathsf{w}_{\alpha,{\ell_{\theta_{\alpha}}}}$ defined by 
\begin{align}\label{modulo of our drift}
	\mathsf{w}_{{\alpha,{\ell_{\theta_{\alpha}}}}}(r):=\,r^{\alpha}\,{\ell_{\theta_{\alpha}}(r)}\log^{1/2}(1/r).
\end{align}
In other word, there is an almost sure finite random variable $\eta=\eta(\omega)$, such that for almost all $\omega \in \Omega$ and  for all $0<r\leq \eta(\omega)$, we have 
\begin{equation}\label{estim mod cont}
	\underset{\substack{s,t \in \mathbf{I}\\ |t-s|\leq r}}{\sup}\left\lVert B^{\delta_{\theta_{\alpha}}}(t)-B^{\delta_{\theta_{\alpha}}}(s)\right\rVert \leq \,\mathtt{c}_1\, \mathsf{w}_{\alpha,{\ell_{\theta_{\alpha}}}}(r),
\end{equation}
where $\mathtt{c}_1$ is a universal positive constant. 
\end{theorem}
\begin{proof}
First, we start by considering the function
\begin{equation}\label{mod-cont}
	\widetilde{\mathsf{w}}_{\alpha,{\ell_{\theta_{\alpha}}}}(r)= r^{\alpha}\,{\ell_{\theta_{\alpha}}(r)}\,\log^{1/2}(1/r)\,+\dint_0^r \frac{u^{\alpha}\,{\ell_{\theta_{\alpha}}(u)}}{u[\log 1 / u]^{1 / 2}} d u. 
\end{equation}
It is simple to verify that $\widetilde{\mathsf{w}}_{\alpha,{\ell_{\theta_{\alpha}}}}(\cdot)$ is well defined in a neighborhood of zero with $\lim_{r\rightarrow 0} \widetilde{\mathsf{w}}_{\alpha,{\ell_{\theta_{\alpha}}}}(r)=0$. 
Since $B^{\delta_{\theta_{\alpha}}}$ satisfies \eqref{conseq bnds increm var}, then it follows from \cite[Theorem 7.2.1 p. 304]{Marcus-Rosen} that $\widetilde{\mathsf{w}}_{{\alpha,{\ell_{\theta_{\alpha}}}}}(\cdot)$ is a uniform modulus of continuity of $B^{\delta_{\theta_{\alpha}}}$. That is there exists a constant $\mathtt{c}_{2} \geq 0$ such that
\begin{equation}
	\limsup_{\eta \rightarrow 0} \sup _{\substack{|u-v| \leq \eta \\ u, v \in I}} \frac{\lVert B^{\delta_{\theta_{\alpha}}}(u)-B^{\delta_{\theta_{\alpha}}}(v)\rVert}{{\widetilde{\mathsf{w}}_{\alpha,{\ell_{\theta_H}}}}(\eta)} \leq \mathtt{c}_{2} \quad a . s .
\end{equation}
Hence there exists an almost surely positive random variable  $\eta_0$ such that for all $0<\eta <\eta_0$, we have 
\begin{equation}\label{mod cont 1}
	\sup _{\substack{|u-v| \leq \eta \\ u, v \in I}} \left\lVert B^{\delta_{\theta_{\alpha}}}(u)-B^{\delta_{\theta_{\alpha}}}(v)\right\rVert \leq \mathtt{c}_{2}\, \widetilde{\mathsf{w}}_{\alpha,{\ell_{\theta_H}}}(\eta).
\end{equation}
The presence of an integral in \eqref{mod-cont} suggests that the modulus of continuity is artificial, which leads to seek a simpler and more practical one. So it is easy to check, by using Hopital's rule argument, that
$$
\dint_0^r \frac{u^{\alpha}\,{\ell_{\theta_H}(u)}}{u(\log 1 / u)^{1 / 2}} d u= o\left( r^{\alpha}\,{\ell_{\theta_H}(r)} \right).
$$
Therefore there exists $r_0>0$ such that 
$$
\widetilde{\mathsf{w}}_{\alpha,L^{-1/2}_{\theta_{\alpha}}(1/\cdot)}(r)\leq 2\,r^{\alpha}\,{\ell_{\theta_H}(r)}\,\log^{1/2}(1/r)=2\, \mathsf{w}_{\alpha,{\ell_{\theta_H}}}(r)\quad \text{ for all $r<r_0$}.$$ 
Hence using this fact and \eqref{mod cont 1}, we obtain that almost surely
\begin{equation*}
	\sup _{\substack{u, v \in \mathbf{I} \\ |u-v| \leq \eta }} \left\lVert B^{\delta_{\theta_{\alpha}}}(u)-B^{\delta_{\theta_{\alpha}}}(v)\right\rVert \leq 2\,\mathtt{c}_{2}\,\mathsf{w}_{\alpha,{\ell_{\theta_H}}}(\eta) \quad \text{ for all $\eta <\eta_0 \wedge r_0$},
\end{equation*}
which completes the proof.
\end{proof}
\begin{remark}
It is noteworthy that \eqref{equiv increm var} ensures $C^{\mathsf{w}_{\delta_{\theta_{\alpha}}}}(\mathbf{I}) \equiv C^{\mathsf{w}_{\alpha,{\ell_{\theta_H}}}}(\mathbf{I})$ where $\mathsf{w}_{\delta_{\theta_{\alpha}}}(\cdot)=\delta_{\theta_{\alpha}}(\cdot)\ log^{1/2}(1/\cdot)$. 
\end{remark}

\subsection{Appendix B} 
Our aim here is to show the existence of probability measures $\nu_1$ and $\nu_2$ supported on two Borel sets $E_1$ and $E_2$ respectively, and satisfying \eqref{modif Dav-Ahlf reg 1} and \eqref{modif Dav-Ahlf reg 2}.
First, let $\varphi$ be a continuous increasing function on $\mathbb{R}_+$, such that
\begin{equation}\label{cond on phi}
\varphi(0)=0\quad \text{and} \quad \varphi(2x)< 2\varphi(x)\quad \text{for all $x\in (0,x_0)$},
\end{equation}
for some 
$x_0\in (0,1)$.
Our aims is to construct a Cantor type set $E_{\varphi}\subset (0,x_0)$ which support a probability measure $\nu_{\varphi}$ with the property 
$\nu_{\varphi}([a-r,a+r])\asymp \varphi(r)$. 
\begin{proposition}\label{ex of modif Dav-Ahlf reg set}
Let $\varphi$ be a function satisfying \eqref{cond on phi}. Then there exists a Borel set $E_{\varphi}\subset [0,x_0]$ which support a probability measure $\nu$ such that 
\begin{equation}
	tc_1^{-1}\,\varphi(r)\leq \nu([a-r,a+r])\leq c_1\,\varphi(r) \quad \text{for all $r\in [0,x_0]$ and\, $a\in E_{\varphi}$}.
\end{equation}
\begin{proof} We will construct a compact set $E_{\varphi}$ inductively as follows: Let $I_0\subset [0,x_0)$ be a closed interval of length ${l}_0<x_0$. First, let $l_1:=\varphi^{-1}\left(\varphi(l_0)2^{-1}\right)$, and let $I_{1,1}$ and $I_{1,2}$ two subintervals of $I_0$ with length $l_1$. For $k\geq 2$, we construct inductively a family of intervals $\{I_{k,j}:j=1,...,2^{k}\}$, in the following way: Let $l_k:=\varphi^{-1}\left(\varphi(l_0)2^{-k}\right)$ and the intervals $I_{k,1},...,I_{k,2^k}$ are constructed by keeping two intervals of length $l_k\,$ from each interval $I_{k-1,i}$ $i=1,\ldots,2^{k-1}$ of the previous iteration. We define $E_{\varphi,k}$ to be the union of the intervals $\left(I_{k,j}\right)_{j=1,...,2^k}$ of each iteration. The compact set $E_{\varphi}$ is defined to be the limit set of this construction, namely 
	\begin{equation}
		E_{\varphi}:=\bigcap_{k=1}^{\infty}E_{\varphi,k}.
	\end{equation}
	Now we define a probability measure $\nu$ on $E_{\varphi}$, by the mass distribution principle \cite{Falconer}. Indeed, for any $k\geq 1$ let us define
	\begin{equation}\label{mass distr}
		\nu(I_{k,i})=2^{-k} \quad \text{ for $i=1,\ldots,2^k$}
	\end{equation}
	and $\nu\left([0,1] \setminus E_{\varphi,k}\right)=0$. Then by \cite[Proposition 1.7]{Falconer}, $\nu$ is a probability measure supported on $E_{\varphi}$. For $a\in E_{\varphi}$ and $0<\eta<\varphi(l_0)$ small enough. Let $k$ be the smallest integer such that $\varphi(l_0)\,2^{-(k+1)}<\eta\leq \varphi(l_0)\,2^{-k}$, then it is not hard to check that the interval $[a-\varphi^{-1}(\eta),a+\varphi^{-1}(\eta)]$ intersects at most $3$ intervals $I_{k,i}$, and contains at least one interval $I_{k+1,j}$ 
	.
	Therefore, using \eqref{mass distr} we obtain
	\begin{equation}
		(1/2\varphi(l_0))\,\eta
		\leq\nu\left([a-\varphi^{-1}(\eta),a+\varphi^{-1}(\eta)]\right)\leq  
		(6/\varphi(l_0))\,\eta.
	\end{equation}
	By making a change of variable $r:=\varphi^{-1}(\eta)$, we get the desired result.
\end{proof}
\end{proposition}
\noindent Now, we can remark that examples of measures ${\nu}_1$ and $\nu_2$ satifying \eqref{modif Dav-Ahlf reg 1} and  \eqref{modif Dav-Ahlf reg 2} respectively, could be deduced from Proposition \ref{ex of modif Dav-Ahlf reg set} with the functions $\varphi_1(r):=r^{\alpha}\,\log^{\beta}(e/r)$ and $\varphi_2(r):=r^{\alpha}\,\log^{-\beta}(e/r)$ for $\alpha \in (0,1)$ and $\beta>1$.

\end{document}